\documentclass[onecolumn,11pt,aps,nofootinbib,superscriptaddress,tightenlines,showkeys]{revtex4}

\usepackage[T1]{fontenc}
\usepackage[utf8]{inputenc}

\usepackage[cmex10]{amsmath}
\usepackage{graphics}
\usepackage{dsfont}
\usepackage{amssymb}
\usepackage{graphicx}
\usepackage{mathrsfs}
\usepackage{units}
\usepackage{pgfplots}
\usepackage{enumerate}
\usepackage{mathdots}
\usepackage[colorlinks=true,linkcolor=black,citecolor=black,plainpages=false,pdfpagelabels]{hyperref}
\hypersetup{pdftitle={Eigenvalue estimates for the resolvent of a non-normal matrix}}

\usepackage{amsthm}
\newtheorem{thm}{Theorem}[section]
\newtheorem{lem}[thm]{Lemma}
\newtheorem{prop}[thm]{Proposition}

\newcommand{\abs}[1]{\ensuremath{|#1|}}

\newcommand{\Abs}[1]{\ensuremath{\left|#1\right|}}

\newcommand{\norm}[2]{\ensuremath{|\!|#1|\!|_{#2}}}

\newcommand{\Norm}[2]{\ensuremath{\left|\!\left|#1\right|\!\right|_{#2}}}

\newcommand{\cC}{\mathcal{C}}

\newcommand{\cK}{\mathcal{K}}

\newcommand{\cM}{\mathcal{M}}

\newcommand{\cR}{\mathcal{R}}

\newcommand{\cT}{\mathcal{T}}

\begin{document}

\title{Optimal estimates for condition numbers and norms of resolvents in
terms of the spectrum }

\author{Oleg Szehr}

\email{oleg.szehr@posteo.de}

\affiliation{Centre for Quantum Information, University of Cambridge,
Cambridge CB3 0WA, United Kingdom}

\author{Rachid Zarouf}

\email{rachid.zarouf@univ-amu.fr}

\affiliation{Aix-Marseille Universit\'e, CNRS, Centrale Marseille, Institut de Math\'ematiques de Marseille (I2M), UMR 7353,
13453 Marseille, France.}

%\keywords[2000]{}

\keywords{Condition number, Resolvent, Toeplitz matrix, Model
matrix, Blaschke product
\\
{2010 Mathematics Subject Classification: Primary: 15A60; Secondary:30D55 }}

\begin{abstract}
In numerical analysis it is often necessary to estimate the condition
number $CN(T)=\Norm{T}{} \cdot\Norm{T^{-1}}{} $
and the norm of the resolvent $\Norm{(\zeta-T)^{-1}}{}$
of a given $n\times n$ matrix $T$. We derive new spectral
estimates for these quantities and compute explicit matrices that achieve our bounds. We recover the well-known fact that
the supremum of $CN(T)$ over all matrices with $\Norm{T}{} \leq1$ and minimal absolute eigenvalue
$r=\min_{i=1,...,n}\Abs{\lambda_{i}}>0$ is the Kronecker bound
$\frac{1}{r^{n}}$. This result is subsequently generalized by computing the corresponding supremum of $\Norm{(\zeta-T)^{-1}}{}$ for any $\Abs{\zeta}\leq1$. We find that the supremum is attained by a triangular Toeplitz matrix. This provides a simple class of structured matrices on which condition numbers and resolvent norm bounds can be studied numerically. The occuring Toeplitz matrices are so-called model matrices, i.e.~matrix representations of the compressed backward shift operator on the Hardy space $H_2$ to a finite-dimensional invariant subspace.
\end{abstract}
\maketitle
\tableofcontents
\section{\label{sub:Introduction}Introduction}

%The purpose of this article is to derive new spectral estimates for the condition number and the norm of the resolvent of a given matrix. 
Let $\cM_n$ be the set of complex $n\times n$ matrices and let $\Norm{T}{}$ denote the spectral norm of $T\in\cM_n$. We denote by $\sigma=\sigma(T)$ the spectrum of $T$ and by $m_{T}$ its minimal polynomial. We denote by $\abs{m_T}$ the degree of $m_T$. In this article we are interested in estimates of the type
\begin{align}\Norm{R(\zeta,\,T)}{}\leq\Phi(\abs{m_{T}},\sigma,\zeta)\label{basic}\end{align}
where $R(\zeta,\,T)=(\zeta-T)^{-1}$ denotes the resolvent of $T$ at point $\zeta\in\mathbb{C}-\sigma$ and $\Phi$ is a function of $\abs{m_{T}}$, $\sigma$ and $\zeta$.
We present a general method that yields {optimal} estimates of the above type and allows us to pinpoint the \lq\lq{}worst\rq\rq{} matrices for this problem.  As it turns out for any $\zeta\not\in\sigma$ among these matrices are so-called analytic Toeplitz matrices. Generally, Toeplitz matrices are characterized
by the existence of a sequence of complex numbers $a=\left(a_{k}\right)_{k=-n+1}^{k=n-1}$
with

\[
T=\left(\begin{array}{ccccc}
a_{0} & a_{-1} & . & . & a_{-n+1}\\
a_{1} & . & . & . & .\\
. & . & . & . & .\\
. & . & . & . & a_{-1}\\
a_{n-1} & . & . & a_{1} & a_{0}
\end{array}\right).
\]
We call $T_{a}$ an \emph{analytic} Toeplitz matrix if $a_{k}=0$
for all $k<0$ and we will denote by $\cT_{n}\subset\cM_{n}$ the set of Toeplitz matrices and by $\cT_{n}^a\subset\cT_{n}$ the analytic Toeplitz matrices.

%r=p(\zeta,\,\sigma):=\min\left\{\vert p(\zeta,\,\lambda)\vert\ :\ \lambda\in\sigma\right\}
%\]
%for the pseudo-hyperbolic distance between $\zeta$ and the spectrum of $T$.
%
%
%
%

We will think of an estimate of the above type as an extremal problem. That is, under certain constraints on the set of admissible matrices $T$ and fixed $\zeta$, we are looking to maximize $\Norm{R(\zeta,\,T)}{}$ over this set. To guarantee a finite bound we at least need a constraint on the $\emph{norm}$ and the $\emph{spectrum}$ of possible $T$. In our discussion we will generally impose that $\Norm{T}{}\leq1$ and call such $T$ a \textit{contraction}. Note that this condition can always be achieved by normalization $T/\Norm{T}{}$. $\cC_{n}\subset\cM_{n}$ denotes the set of contractions. To ensure that the resolvent is finite, $\zeta$ must be separated from $\sigma$. Depending on the situation it will be convenient to measure this separation in \emph{Euclidean} distance
$$d(z,\, w):=\Abs{z-w}\ ,\ z, w\in\mathbb{C}$$
or \emph{pseudo-hyperbolic} distance
$$p(z,\, w):=\Abs{\frac{z-w}{1-\overline{z}w}},\:\: z,\, w\in\mathbb{C}.$$
%We write shortly  \[

The case $\zeta=0$ in \eqref{basic} corresponds to the study of the \textit{condition number} $CN(T)=\Norm{T}{} \cdot\Norm{T^{-1}}{}$. It was shown by Kronecker in the 19th century that 
\begin{equation}
\sup\left\{\Norm{T^{-1}}{}\ :\ T\in\cC_{n},\ p(0,\,\sigma)=r\right\}=\frac{1}{r^{n}}.\label{eq:Kroneck_result}
\end{equation}
The upper bound can be seen using polar decomposition~\cite[p.~137]{NN1}. To prove the reverse inequality, one can use A. Horn's theorem (see \cite{AH} and also \cite[Chap. 9, Sect. E]{MO}) that gives a converse to the famous Weyl inequalities~\cite[p.~137]{NN1}.  

The case $\Abs{\zeta} =1$ was studied by Davies and Simon and applied to characterize the localization of zeros of random orthogonal polynomials on the unit circle~\cite{DS}. The corresponding estimate for $\abs{\zeta}=1$ is
\begin{equation}
\sup\left\{{d}(\zeta,\,\sigma)\Norm{ R(\zeta,\,T)}{}\ :\ T\in\cC_{n}\right\}=\cot(\frac{\pi}{4n}).\label{eq:DS_result}
\end{equation}
Note that $\Abs{\zeta} =1$ implies that we have $p(\zeta,\,\sigma)=1$  for any $\sigma$ such that we could add this condition to the optimization problem. In this paper we provide an approach, which

1) allows us to treat the problems (\ref{eq:Kroneck_result}) and (\ref{eq:DS_result})
simultaneously and generalize them to any $\Abs{\zeta} \leq1$,

2) strengthens the result (\ref{eq:DS_result}) for $\Abs{\zeta} =1$,
by considering the $\sup$ over $T\in\mathcal{C}_{n}$ with fixed minimal polynomial and

3) provides explicit analytic Toeplitz matrices that achieve the estimates~\eqref{basic}, (\ref{eq:Kroneck_result}) and (\ref{eq:DS_result}).

Already for $\zeta=0$ it is a non-trivial question which matrices achieve the estimate (\ref{eq:Kroneck_result}). A theoretic classification of such matrices was given in \cite{NN1}. However, it seems rather difficult to compute explicit matrix representations from this classification. It was suspected in \cite{RZ} that the estimate (\ref{eq:Kroneck_result}) is not reached by any Toeplitz matrix; this claim is refuted here. 

\section{\label{sec:Main-results} Resolvent bounds}

Before we state our main results we need to introduce some further notation. We denote by $\mbox{\ensuremath{\mathbb{D}}}=\left\{ z\in\mathbb{C}:\:\vert z\vert<1\right\}$
the open unit disc in the complex plane, $\overline{\mathbb{D}}$ is its closure and $\partial\mathbb{D}$ its boundary.
For any finite set $\sigma\subset\mathbb{D}$
we define
\[
s(\zeta,\,\sigma):=\max\left\{\frac{1-\Abs{\lambda}^{2}}{\vert1-\bar{\lambda}\zeta\vert}:\:\lambda\in\sigma,\: p(\zeta,\,\lambda)=p(\zeta,\,\sigma)\right\}.
\]
The function $s$ is related to the Stolz angle~\cite[p.~23]{JC} between  
$\zeta\in\overline{\mathbb{D}}-\sigma$ and $\lambda\in\sigma$.
%between $\zeta\in\overline{\mathbb{D}}-\sigma$ and $\lambda\in\sigma$ on a hyperbolic circle of radius $p(\zeta,\sigma)$ around $\zeta$.
For $T\in\cC_{n}$ the function $\zeta\mapsto s(\zeta,\,\sigma(T))$ is bounded on $\overline{\mathbb{D}}$ by $1+\rho(T)\leq2$, where
\[
\rho(T):=\max\{\Abs{\lambda}:\:\lambda\in\sigma(T)\}
\]
is the spectral radius of $T$.
Of particular importance to our discussion will be the analytic Toeplitz matrix $X_{r,\,\beta}$ given entry-wise by
\begin{equation}
\left(X_{r,\,\beta}\right)_{ij}=\Bigg\{ \begin{array}{ll}
0 & \mbox{if }i<j\\
r^{n-1} & \mbox{if }i=j\\
\beta r^{n-(i-j+1)} & \mbox{if }i>j
\end{array}\label{eq:X_xi_beta}
\end{equation}
with $r\in[0,1]$ and $\beta\in[0,2]$. The spectral norm of $X_{r,\,\beta}$ is computed below in Proposition~\ref{Th_3_computation_of_norm_of_Xtilde}.

\subsection{Optimal upper bounds}
\begin{thm}
\label{Th_upper_bound} Let $T\in\cC_{n}$ with minimal polynomial
$m$ and spectrum $\sigma$. For any $\zeta\in\overline{\mathbb{D}}-\sigma$,
it holds for the resolvent of $T$ that
\begin{equation}
\Norm{R(\zeta,\,T)}{}\leq\frac{1}{d(1,\bar{\sigma}\zeta)}\frac{1}{r^{\Abs{m}}}\Norm{X_{r,\,\beta}}{},\label{eq:gen_upp_bd}
\end{equation}
where $X_{r,\,\beta}$ is the $\Abs{m}\times\Abs{m}$ analytic
Toeplitz matrix defined in \eqref{eq:X_xi_beta}, $r=p(\zeta,\,\sigma)$
and $\beta=s(\zeta,\,\sigma)$.  Moreover, for any $\zeta\in[-1,\,1]$
and $\lambda\in(-1,\,1)$, $\lambda\neq\zeta$ the analytic $n\times n$ Toeplitz matrix  
\[
T^*=\left(\begin{array}{ccccc}
\lambda & 0 & \ldots & \ldots & 0\\
1-\lambda^{2} & \lambda & \ddots & . & \vdots\\
-\lambda(1-\lambda^{2}) & 1-\lambda^{2} & \lambda & \ddots & \vdots\\
\vdots & \ddots & \ddots & \ddots & 0\\
(-\lambda)^{n-2}(1-\lambda^{2}) & \ldots & -\lambda(1-\lambda^{2}) & 1-\lambda^{2} & \lambda
\end{array}\right),
\]
is a contraction of minimal polynomial $m=(z-\lambda)^{n}$ that achieves \eqref{eq:gen_upp_bd}.  

\end{thm}
\textbf{Remarks:}
\begin{enumerate}
\item  Setting $\zeta=0$ in Theorem~\ref{Th_upper_bound} we recover Kronecker's
result for condition numbers. Moreover, an analytic Toeplitz
matrix is extremal for the Kronecker maximization problem~\eqref{eq:Kroneck_result}.
This answers the questions raised in \cite{RZ}. 
\item Considering Theorem~\ref{Th_upper_bound} for
$\zeta\in\partial\mathbb{D}$, and bounding
\[
\Norm{X_{1,\,s(\zeta,\sigma(T))}}{}\leq\Norm{X_{1,\,1+\rho(T)}}{}\leq\Norm{X_{1,\,2}}{},
\]
we recover the resolvent estimate of~\cite{DS} as $\Norm{X_{1,\,2}}{}=\cot(\frac{\pi}{4n})$ 
(cf.~Proposition \ref{Th_3_computation_of_norm_of_Xtilde}
below). However, if e.g.~$\rho(T)$ is known we can improve on~\cite{DS}.
\item Equality is achieved by a so-called \emph{model matrix} corresponding to a fully degenerate spectrum. This matrix arises as a representation of the compressed backward shift operator of $H_2$ with respect to the 
\textit{Malmquist-Walsh basis}, see Section \ref{sub:Model-spaces-and} below for details. The model matrix is a natural generalization of a simple Jordan block of size $n$ to the situation when the spectrum is $\sigma=\left\{a\right\}\subset\overline{\mathbb{D}}$.
\end{enumerate}
The next proposition treats the case $\sigma(T)\subset\partial\mathbb{D}$.
\begin{prop}
\label{prop:unimod_spec}Let $T\in\cC_{n}$ with
spectrum $\sigma\subset\partial\mathbb{D}$. Then for any $\zeta\in\mathbb{C}-\sigma$
\begin{equation}
\Norm{R(\zeta,\,T)}{}\leq\frac{1}{d(\zeta,\,\sigma)}.\label{eq:unimod_spect}
\end{equation}
\end{prop}
It follows that if $T$ is a contraction such that
$\sigma(T)\subset\partial\mathbb{D}$ then it satisfies a
\textit{Linear Resolvent Growth} with constant 1, see \cite{BN}. %
Clearly, any diagonal matrix in  $\cC_{n}$ achieves the inequality (\ref{eq:unimod_spect}).

Theorem~\ref{Th_upper_bound} has the disadvantage that the upper bound depends on the slightly unyielding quantity $\Norm{X_{r,\beta}}{}$. Taking together Theorem~\ref{Kronecker_type_theorem} with Proposition~\ref{Th_3_computation_of_norm_of_Xtilde} (both below) leads directly to the following.

\begin{thm}\label{mtt}
{Let $T\in\cC_n$ with minimal polynomial $m$ and spectrum $\sigma$. For any $\zeta\in{\mathbb{D}}-\sigma$ let $p(\zeta,\sigma)=r$ denote the pseudo-hyperbolic distance of $\zeta$ to $\sigma$. Then it holds for the resolvent of $T$ that}
\begin{align*}
\Norm{R(\zeta,T)}{}\leq\frac{1}{d(1,\bar\sigma\zeta)}\frac{1}{r^{\abs{m}}(1-r\abs\zeta)}.
\end{align*}
The estimate is sharp in the limit of large $\abs{m}$.
\end{thm}

\subsection{Constrained maximization of the resolvent} \label{indadev}

For any given $\zeta\in\overline{\mathbb{D}}$ and $r\in[0,1]$ we study the quantity
\begin{equation}
\cR(\zeta,\, r):=\sup\left\{d(1,\,\bar{\sigma}\zeta)\Norm{ R(\zeta,\,T)}{}\ :\ T\in\cC_{n},\ p(\zeta,\sigma)=r\right\},\label{eq:Kr_Cst_1}
\end{equation}
where $\bar{\sigma}=\{\bar{\lambda}:\lambda\in\sigma\}$. $\cR$ also depends on $n$ but we keep the notation simple and do not write this explicitly.
\begin{thm}
\label{Kronecker_type_theorem} For any $\zeta\in\mathbb{D}$ and
$r\in(0,\,1)$ we have that
\begin{equation*}
\cR(\zeta,\, r)=\frac{1}{r^{n}}\Norm{X_{r,\beta_{\max}}}{},
\end{equation*}
where $\beta_{\max}=\frac{1-r^{2}}{1-r\Abs{\zeta} }$. The estimate is achieved by an analytic Toeplitz matrix
and in the limit of large matrices we find
\[
\cR(\zeta,\, r)\sim\frac{1}{r^{n}}\frac{1}{1-r\Abs{\zeta} }\quad\textnormal{\emph{as}}\quad  n\rightarrow\infty.
\]
\end{thm}
Recall that Kronecker's result (\ref{eq:Kroneck_result}) is
\[
\cR(0,\, r)=\frac{1}{r^{n}}.
\]
 Davies and Simon (\ref{eq:DS_result}) have shown that
\[
\cR(\zeta,\,1)=\cot(\frac{\pi}{4n}),\qquad\zeta\in\partial\mathbb{D}-\{1\},
\]
which can be seen in the above theorem by continuously moving $\zeta$ to $\partial\mathbb{D}$. Note, however, the fundamental difference in the large $n$ behavior of the resolvent for $\zeta\in\mathbb{D}$ and $\zeta\in\partial\mathbb{D}$. This is a consequence of the fact that for $\zeta$ inside the unit disk $\norm{X_{r,\beta}}{}$ is bounded from above, see Proposition~\ref{Th_3_computation_of_norm_of_Xtilde}.
%
%

%Relatively to the above points 1), 2) and 3) from ,
%we respectively

Finally for $\zeta\in\partial\mathbb{D}$ and a polynomial $P_{\sigma}$ with zero set $\sigma$ so that $p(\zeta,\sigma)=1$ we can compute the constant
\begin{equation}
\cK(\zeta,\,P_{\sigma})=\sup\left\{d(\zeta,\,\sigma)\Norm{R(\zeta,\,T)}{}\ :\ T\in\cC_{n},\ m_{T}=P_{\sigma}\right\},\label{eq:DS_constant}
\end{equation}
which is bounded by $\cR(\zeta,\,r)$.
\begin{prop}
\label{cor:Davies_Simon_type_th} Let $P_{\sigma}$ be a polynomial
with zero set $\sigma=\sigma_{1}\cup\sigma_{2}$, $\sigma_{1}\subset\mathbb{D}$
and $\sigma_{2}\subset\partial\mathbb{D}$ so that $P_{\sigma}$=$P_{\sigma_{1}}P_{\sigma_{2}}.$
For any $\zeta\in\partial\mathbb{D}-\sigma$, we have 
\begin{align}
\cK(\zeta,\, P_{\sigma}) & \leq\max(1,\,\Norm{X_{1,\, s(\zeta,\,\sigma_{1})}}{})\leq\Norm{X_{1,\,2}}{}=\cot(\frac{\pi}{4\abs{P_{\sigma_{1}}}}),\label{eq:th4_majoration}
\end{align}
where $X_{1,\, s(\zeta,\,\sigma)},\ X_{1,\,2}\in\cM_{\abs{P_{\sigma_{1}}}}(\mathbb{C})$.
Moreover, fixing $\zeta\in\partial\mathbb{D}-\sigma,$ $n_{1},\, n_{2}\geq1$
and $0\leq\rho_1<1$ and taking the supremum of $\cK(\zeta,\, P_{\sigma})$
over all sets $\sigma=\sigma_{1}\cup\sigma_{2}$ with $n_{1}=\abs{P_{\sigma_{1}}}$
and $n_{2}=\abs{P_{\sigma_{2}}}$ such that $\max_{\lambda\in\sigma_{1}}\abs{\lambda}\leq\rho_1$
we get: 
\begin{equation}
\sup\cK(\zeta,\, P_{\sigma})=\Norm{X_{1,\,1+\rho_1}}{}\label{eq:th4_egalite}
\end{equation}
where $X_{1,\,1+\rho_1}\in\cM_{n_{1}}(\mathbb{C})$. In particular,
the supremum of the above quantity taken over $0\leq\rho_1<1$ is equal
to 
\[
\cot(\frac{\pi}{4n_{1}}).
\]
\end{prop}
\subsection{Computing $\Norm{X_{r,\,\beta}}{}$}

In this subsection we study the quantity $\norm{X_{r,\,\beta}}{}$ for $r\in(0,1]$ and $\beta\in[0,2]$. We provide an implicit formula for $\Norm{X_{r,\,\beta}}{}$.

\begin{prop}
\label{Th_3_computation_of_norm_of_Xtilde} Let $X_{r,\,\beta}$ with $r\in(0,1]$ and $\beta\in[0,2]$ be the $n\times n$ matrix introduced in \eqref{eq:X_xi_beta}. The following assertions hold.
\begin{enumerate}[i)]

\item If $\beta=1-r^2$ then $\Norm{X_{r,\,\beta}}{}=1$.

\item At fixed $n$, $\Norm{X_{r,\,\beta}}{}$ grows monotonically with
$r$ and $\beta$.

\item At fixed $r$ and {$1-r^2\leq\beta$}, $\Norm{X_{r,\,\beta}}{}$ grows monotonically
with $n$.

\item Let $U_n$ denote the $n$-th degree Chebychev
polynomial of second kind, i.e.~$U_n(cos(\theta))=\frac{\sin{((n+1)\theta)}}{\sin{(\theta)}}$. If $\beta\neq1-r^{2}$ then $\Norm{X_{r,\,\beta}}{}=\abs{\lambda^*}$,
where $\lambda^*$ is the largest in magnitude solution of the polynomial equation
\begin{equation}
\mu^{n+1}r U_n\left(\frac{\gamma}{2\mu}\right)+\mu^n(\lambda^2-r^{2n-2}(\beta-1)^2)U_{n-1}\left(\frac{\gamma}{2\mu}\right)=0\label{eq:rrr}
\end{equation}
with $\mu=\lambda^2+r^{2n-2}(\beta-1)$ and $\gamma=-\lambda^2(r+1/r)+r^{2n-2}(r+\frac{1}{r}(\beta-1)^{2})$.

\item If $r\in(0,1)$ and $1-r^2\leq\beta$ we have the limit
\begin{equation*}
\Norm{X_{r,\,\beta}}{}\rightarrow\frac{\beta}{1-r^{2}},\qquad n\rightarrow\infty.\label{eq:limit}
\end{equation*}
\item{For any $n$, $r\in(0,1)$, $\beta_{\max}\in(0,2]$ and $1-r^2\leq\beta\leq\beta_{\max}$ we have}
\begin{align*}
\Norm{X_{r,\,\beta}}{}\leq\frac{\beta_{\max}}{1-r^2}.
\end{align*}
\end{enumerate}
\end{prop}
The equations in \emph{iv)} (see~\eqref{eq:rrr}) are discussed in Lemma~\ref{lem3} below.
For instance we can recover known values of $\Norm{X_{1,\,\beta}}{}$ from~Proposition~\ref{Th_3_computation_of_norm_of_Xtilde},~\emph{iv)}. In this case
\begin{equation*}
\Norm{X_{1,\,\beta}}{}=\frac{1}{2}\sqrt{(\beta-2)^{2}+\frac{\beta^{2}}{\cot(\theta^{*}/2)}},\label{eq:formula_norm_zeta_1}
\end{equation*}
where $\theta^{*}$ is the unique solution of \eqref{eq:rrr} in $[\frac{2n-1}{2n}\pi,\,\pi)$ and it follows, see~\cite{OS}, that
\begin{equation*}
\Norm{X_{1,\,0}}{}=1,\qquad\Norm{X_{1,\,1}}{}=\frac{1}{2\sin(\frac{\pi}{4n+2})},\qquad\Norm{X_{1,\,2}}{}=\cot(\frac{\pi}{4n}).%\label{eq:norm_3cases_zeta_1}
\end{equation*}

\begin{lem}
\label{lem3}Let $r\in(0,1]$ and $\beta\in(0,2]$ such that $\beta-1+r^{2}\neq0$.
We consider the equation 

\begin{equation}
\cot(n\theta)+\frac{(2-\beta)r}{r^{2}+(\beta-1)}\frac{1}{\sin(\theta)}+\frac{r^{2}-(\beta-1)}{r^{2}+(\beta-1)}\cot(\theta)=0.
\label{eq:rrrr}
\end{equation}

\begin{enumerate}
\item If $r=1$, then \eqref{eq:rrrr} has exactly $2n$ distinct solutions
in $[-\pi,\,\pi)$. 
\item If $r=\Abs{\beta-1}\notin\{0,\,2\}$, we distinguish two subcases:

\begin{enumerate}
\item if $\beta<1$, then \eqref{eq:rrrr} has exactly $2n-2$ distinct solutions
in $[-\pi,\,\pi)$ for any $n\geq1$, 
\item if\textbf{ $\beta>1$}, then \eqref{eq:rrrr} has exactly $2n-2$ distinct
solutions in $[-\pi,\,\pi)$ if and only if $n>\frac{\beta}{2(2-\beta)}$. If $n\leq\frac{\beta}{2(2-\beta)}$, then  \eqref{eq:rrrr} has exactly $2n$ distinct
solutions in $[-\pi,\,\pi)$.
\end{enumerate}
\item If $\beta>1$, and $r\notin\{1,\,\beta-1\}$, we distinguish two subcases:

\begin{enumerate}
\item if $1<\beta<1+r$, then \eqref{eq:rrrr} has exactly $2n-2$ distinct
solutions in $[-\pi,\,\pi)$ if and only if $n>\frac{\beta-1+r^{2}}{(1-r)(r+\beta-1)}$. If $n\leq\frac{\beta-1+r^{2}}{(1-r)(r+\beta-1)}$, then  \eqref{eq:rrrr} has exactly $2n$ distinct solutions in $[-\pi,\,\pi)$.
 
\item If\textbf{ $\beta>1+r$}, and $n\geq2$, then \eqref{eq:rrrr} has exactly $2n-4$
distinct solutions in $[-\pi,\,\pi)$ if and only if $n>\frac{\beta-1+r^{2}}{(1-r)(r+\beta-1)}$.  If $\frac{\beta-1+r^{2}}{(1+r)(\beta-1-r)} < n \leq\frac{\beta-1+r^{2}}{(1-r)(r+\beta-1)}$, then \eqref{eq:rrrr} has exactly $2n-2$ distinct solutions in $[-\pi,\,\pi)$. 
Finally, if $n\leq\frac{\beta-1+r^{2}}{(1+r)(\beta-1-r)}$, then  \eqref{eq:rrrr} has exactly $2n$ distinct solutions in $[-\pi,\,\pi)$.

\end{enumerate}
\item If $\beta<1$, and $r\notin\{1,\,1-\beta\}$, we distinguish two subcases:

\begin{enumerate}
\item if $1-r^{2}<\beta<1$, then \eqref{eq:rrrr} has exactly $2n-2$ distinct
solutions in $[-\pi,\,\pi)$ if and only if $n>\frac{\beta-1+r^{2}}{(1-r)(r+\beta-1)}$, 
\item if $0<\beta<1-r^{2}$, we distinguish two subcases:

\begin{enumerate}
\item if $0<\beta<1-r$, then \eqref{eq:rrrr} has exactly $2n-4$ distinct
solutions in $[-\pi,\,\pi)$ if and only if $n>\frac{\beta-1+r^{2}}{(1-r)(r+\beta-1)}$. If $n\leq\frac{\beta-1+r^{2}}{(1-r)(r+\beta-1)}$, then  \eqref{eq:rrrr} has exactly $2n$ distinct solutions in $[-\pi,\,\pi)$.

\item If $1-r<\beta<1-r^{2}$, then \eqref{eq:rrrr} has exactly $2n-2$
distinct solutions on $[-\pi,\,\pi)$ for any $n\geq2$. 
\end{enumerate}
\end{enumerate}
\item If $\beta=1$ then \eqref{eq:rrrr} has exactly $2n-2$ distinct solutions
in $[-\pi,\,\pi)$ if and only if $n>\frac{r}{1-r}$. If $n\leq\frac{r}{1-r}$, then  \eqref{eq:rrrr} has exactly $2n$ distinct
solutions in $[-\pi,\,\pi)$.

\end{enumerate}
\end{lem}

\section{\label{sub:Model-spaces-and}Model spaces and Model operators}

This section lays down the theoretical framework on which our results
are footed. We refer to \cite{NN3} for an in-depth study of the
topic. To the minimal polynomial $m=m_{T}$ of $T\in\cM_{n}$ we associate
a \emph{Blaschke product} 
\begin{align*}
B(z):=\prod_{i}\frac{z-\lambda_{i}}{1-\bar{\lambda}_{i}z}.
\end{align*}
The product is taken over all $i$ such that the corresponding linear
factor $z-\lambda_{i}$ occurs in the minimal polynomial $m$. Thus,
the numerator of $B$ as defined here is exactly the associated minimal
polynomial. The space of holomorphic functions on $\mathbb{D}$ is denoted
by $Hol(\mathbb{D})$. The Hardy spaces considered here are 
\begin{align*}   
H_{2}:=\big\{ f\in Hol(\mathbb{D}):\:\Norm{f}{H_{2}}^{2}:=\sup_{0\leq r<1}\frac{1}{2\pi}\int_{0}^{2\pi}\vert f(re^{i\phi})\vert^{2}{\rm d}\phi<\infty\big\},
\end{align*}
and 
\begin{align*}
 & H_{\infty}:=\big\{ f\in Hol(\mathbb{D}):\:\Norm{f}{H_{\infty}}:=\sup_{z\in\mathbb{D}}\vert f(z)\vert<\infty\big\}.
\end{align*}
Let $B$ be the Blaschke product associated to the minimal polynomial
of $T\in\cM_{n}$. We define the $\Abs{m}$-dimensional
\emph{model space} 
\begin{align*}
K_{B}:=H_{2}\ominus BH_{2}:=H_{2}\cap(BH_{2})^{\bot},
\end{align*}
where we employ the usual scalar product from $L^{2}(\partial\mathbb{D})$. The \emph{model operator} $M_{B}$ acts on $K_{B}$ as 
\begin{align*}
M_{B}:\: K_{B} & \rightarrow K_{B}\\
f & \mapsto M_{B}(f):=P_{B}(zf),
\end{align*}
where $P_{B}$ denotes the orthogonal projection on $K_{B}$. In other
words, $M_{B}$ is the compression of the multiplication operation
by $z$ to the model space $K_{B}$. Multiplication by $z$ has operator
norm $1$ so that $\Norm{M_{B}}{}\leq1$. Moreover, it is not hard
to show that the eigenvalues of $M_{B}$ are exactly the zeros of
the corresponding Blaschke product, see \cite{NN2}. 

For our subsequent discussion we will assume that $T$ can be diagonalized. This does
not result in any difficulties since the upper bounds obtained in
the special case extend by continuity to bounds for non-diagonalizable
$T$, see~\cite{OS}. One natural orthonormal basis for $K_{B}$ (for diagonalizable $T$) is the Malmquist-Walsh
basis $\{e_{k}\}_{k=1,...,\Abs{m}}$ with (\cite[p. 137]{NN3}) 
\begin{align*}
e_{k}(z):=\frac{(1-\vert\lambda_{k}\vert^{2})^{1/2}}{1-\bar{\lambda}_{k}z}\prod_{i=1}^{k-1}\frac{z-\lambda_{i}}{1-\bar{\lambda}_{i}z}
\end{align*}
with 
\begin{align*}
e_{1}(z)=\frac{(1-\vert\lambda_{1}\vert^{2})^{1/2}}{1-\bar{\lambda}_{1}z}.
\end{align*}

\section{\label{sec:Proofs}Proofs }

\subsection{Proofs of Theorem \ref{Th_upper_bound} and Proposition \ref{prop:unimod_spec}}

We recall~\cite[Theorem III.2]{OS} that given a finite Blaschke product
$B=\prod_{i=1}^{N}\frac{z-\nu_{i}}{1-\bar{\nu}_{i}z}$, of degree
$N\geq1$, with zeros $\nu_{i}\in\mathbb{D}$, the components of the resolvent
of the model operator $M_{B}$ at any point $\zeta\in\mathbb{C}-\{\nu_1,...,\nu_N\}$
with respect to the Malmquist-Walsh basis are given by
\begin{equation}
\left(\big(\zeta-M_{B}\big)^{-1}\right)_{1\leq i,\, j\leq N}=\frac{1}{B(\zeta)}\left\{ \begin{array}{ll}
0 & \mbox{if }i<j\\
\frac{1}{1-\overline{\nu_{i}}\zeta}\prod_{s\neq i}\frac{\zeta-\nu_{s}}{1-\overline{\nu_{s}}\zeta} & \mbox{if }i=j\\
\frac{(1-\vert\nu_{i}\vert^{2})^{1/2}}{1-\overline{\nu_{i}}\zeta}\frac{(1-\vert\nu_{j}\vert^{2})^{1/2}}{1-\overline{\nu_{j}}\zeta}\prod_{s\notin\{j,\,\dots i\}}\frac{\zeta-\nu_{s}}{1-\overline{\nu_{s}}\zeta} & \mbox{if }i>j
\end{array}.\right.\label{eq:OS_formula}
\end{equation}
We begin by proving Proposition~\ref{prop:unimod_spec}.

\begin{proof}[Proof of Proposition~\ref{prop:unimod_spec}] We consider $T\in\cC_n$ and assume for a moment that $\sigma(T)\subset\mathbb{D}$. It follows from \cite[Theorem III.2]{OS} that for any $\zeta\in\mathbb{C}-\sigma(T)$
the resolvent of $T$ is bounded by
\begin{equation*}
\Norm{R(\zeta,\, T)}{}\leq\Norm{R(\zeta,\, M_{B})}{},
\end{equation*}
where $B$ is the Blaschke product associated with $m_T$. We conclude from \eqref{eq:OS_formula} that
\[
\left(\big(\zeta-M_{B}\big)^{-1}\right)_{1\leq i,\, j\leq\Abs{m}}=\left\{ \begin{array}{ll}
0 & \mbox{if }i<j\\
\frac{1}{\zeta-\lambda_{i}} & \mbox{if }i=j\\
\frac{(1-\Abs{\lambda_{i}}^{2})^{1/2}}{1-\overline{\lambda_{i}}\zeta}\frac{(1-\Abs{\lambda_{j}}^{2})^{1/2}}{1-\overline{\lambda_{j}}\zeta}{\prod_{k=j}^{i}\frac{1-\bar{\lambda}_k\zeta}{\zeta-\lambda_k}} & \mbox{if }i>j
\end{array}.\right.
\]
%
%
%(The more general case that $\sigma(T)\subset\bar{\mathbb{D}}$ follows immediately by continuously moving eigenvalues towards $\partial\mathbb{D}$.)
Taking the limit $\abs{\lambda_i}\rightarrow1$ we have that for $\zeta\in\mathbb{C}-\sigma(T)$ the norm of the resolvent is bounded by the norm of the diagonal matrix
\[
\left(\begin{array}{ccccc}
\frac{1}{\zeta-\lambda_{1}}\\
  & \ddots\\
 &  &  \frac{1}{\zeta-\lambda_{n}}
\end{array}\right),
\]
which is $\frac{1}{d(\zeta,\,\sigma)}$.
\end{proof}
\begin{proof}[Proof of Theorem \ref{Th_upper_bound}] 
We begin by showing that \eqref{eq:gen_upp_bd} holds.
As before we suppose $\sigma(T)\subset\mathbb{D}$. The more general case $\sigma(T)\subset\bar{\mathbb{D}}$ follows immediately by continuously moving the eigenvalue towards $\partial\mathbb{D}$. As before we have that  
$
\Norm{R(\zeta,\,T)}{}\leq\Norm{R(\zeta,\,M_{B})}{}
$
and that
\[
\left(\big(\zeta-M_{B}\big)^{-1}\right)_{1\leq i,\, j\leq\Abs{m}}=\left\{ \begin{array}{ll}
0 & \mbox{if }i<j\\
\frac{1}{1-\overline{\lambda_{i}}\zeta}\frac{1}{b_{\lambda_{i}}} & \mbox{if }i=j\\
\frac{(1-\Abs{\lambda_{i}}^{2})^{1/2}}{1-\overline{\lambda_{i}}\zeta}\frac{(1-\vert\lambda_{j}\vert^{2})^{1/2}}{1-\overline{\lambda_{j}}\zeta}\frac{1}{\prod_{k=j}^{i}b_{\lambda_{k}}} & \mbox{if }i>j
\end{array},\right.
\]
where we abbreviated
$b_{\lambda_{i}}=\frac{\zeta-\lambda_{i}}{1-\overline{\lambda_{i}}\zeta}$. Recall that for any $n\times n$ matrices $A=(a_{ij})$ and $A=(a'_{ij})$,
the condition $\vert a_{ij}\vert\leq a'_{ij}$ implies that $\Norm{A}{}\leq\Norm{A'}{}$.
Hence, we estimate
\[
\Abs{\big(\zeta-M_{B}\big)_{i,\, j}^{-1}}\leq\left\{ \begin{array}{ll}
0 & \mbox{if }i<j\\
\frac{1}{\vert1-\overline{\lambda_{i}}\zeta\vert}\frac{1}{\vert b_{\lambda_{i}}\vert} & \mbox{if }i=j\\
\frac{(1-\vert b_{\lambda_{i}}\vert^{2})^{\frac{1}{2}}(1-\vert  b_{\lambda_{j}}\vert^{2})^{\frac{1}{2}}}{1-\Abs{\zeta}^{2}}\frac{1}{\prod_{k=j}^{i}\vert b_{\lambda_{k}}\vert} & \mbox{if }i>j
\end{array},\right.
\]
where we made use of the well-known~\cite{JG} fact that
\[
\frac{1-\Abs{\lambda}^{2}}{\vert1-\bar{\lambda}\zeta\vert^{2}}=\frac{1-\vert b_{\lambda}(\zeta)\vert^{2}}{1-\Abs{\zeta}^{2}},\qquad\lambda\neq\zeta.
\]
By assumption $\min_{i}\vert b_{\lambda_{i}}\vert=p(\zeta,\,\sigma)=r$ and for any
$l$ with $\vert b_{\lambda_{l}}\vert=r$ we have

\[
\Abs{\big(\zeta-M_{B}\big)_{i,\, j}^{-1}}\leq\left\{ \begin{array}{ll}
0 & \mbox{if }i<j\\
\frac{1}{\vert1-\overline{\lambda_{i}}\zeta\vert}\frac{1}{r} & \mbox{if }i=j\\
\frac{1-\vert b_{\lambda_{l}}(\zeta)\vert^{2}}{1-\Abs{\zeta}^{2}}\frac{1}{r^{i-j+1}} & \mbox{if }i>j
\end{array},\right.
\]
and since 
\[
\frac{1-\vert b_{\lambda_{l}}(\zeta)\vert^{2}}{1-\Abs{\zeta}^{2}}=\frac{1-\vert\lambda_{l}\vert^{2}}{\vert1-\overline{\lambda_{l}}\zeta\vert^{2}}\leq\frac{1}{d(1,\bar{\sigma}\zeta)}\frac{1-\vert\lambda_{l}\vert^{2}}{\vert1-\overline{\lambda_{l}}\zeta\vert}\leq\frac{\beta}{d(1,\bar{\sigma}\zeta)},
\]
with $\beta=s(\zeta,\,\sigma)$ we conclude that

\[
\Abs{\big(\zeta-M_{B}\big)_{i,\, j}^{-1}}\leq\frac{1}{d(1,\bar{\sigma}\zeta)}\frac{1}{r^{\Abs{m}}}\left\{ \begin{array}{ll}
0 & \mbox{if }i<j\\
r^{\Abs{m}-1} & \mbox{if }i=j\\
\beta r^{\Abs{m}-(i-j+1)} & \mbox{if }i>j
\end{array}.\right.
\]
Next we discuss the sharpness of \eqref{eq:gen_upp_bd}. We consider $\zeta\in[-1,1]$ and the model operator $M_B$ with $B=b_{\lambda}^{n}$ and $\lambda\in(-1,\,1)$. In this case $d(1,\bar{\sigma}\zeta)=\abs{1-\bar\lambda\zeta}$ and $p(\zeta,\,\sigma)=\abs{b_{\lambda}({\zeta})}$
and $s({\zeta},\,\sigma)=\frac{1-\Abs{\lambda}^{2}}{\abs{1-\bar{\lambda}{\zeta}}}$. We find that
\[
\big({\zeta}-M_{B}\big)^{-1}_{i,\, j}=\frac{1}{d(1,\bar{\sigma}\zeta)}\frac{1}{B({\zeta})}\left\{ \begin{array}{ll}
0 & \mbox{if }i<j\\
(b_{\lambda}({\zeta}))^{n-1} & \mbox{if }i=j\\
\frac{{1-\abs{\lambda}^{2}}}{\abs{1-\bar{\lambda}{\zeta}}}(b_{\lambda}({\zeta}))^{n-(i-j+1)} & \mbox{if }i>j
\end{array}\right..
\]
If $b_{\lambda}({\zeta})>0$ then the equality%
\[
\Norm{R({\zeta},M_{B})}{}=\frac{1}{d(1,\bar{\sigma}\zeta)}\frac{1}{r^{n}}\Norm{X_{r,\,\beta}}{}
\]
is clear. If $b_{\lambda}({\zeta})<0$ equality
still holds because the Toeplitz matrices $a=(a_{k})_{k\geq0}$
and $b=((-1)^{k}a_{k})_{k\geq0}$ have the same norm. To see this write $D=diag(-1,...,-1)$ and note that $DT_{a}D^\dagger=T_{b}$. As $D$ is unitary $\Norm{T_{a}}{}=\Norm{T_{b}}{}$. \end{proof}

\subsection{Proofs of Theorem \ref{Kronecker_type_theorem} and Propositions \ref{cor:Davies_Simon_type_th}, \ref{Th_3_computation_of_norm_of_Xtilde}}

\begin{lem} \label{lemma_real_zeta}For
any $\zeta\in\mathbb{D}$ and $r\in(0,\,1)$
\[
\cR_{n}(\zeta,\, r)=\cR_{n}(\Abs{\zeta},\, r).
\]
\end{lem} \begin{proof} Indeed, writing $\zeta=\Abs{\zeta}\exp(i\theta)$ we have $
\Abs{1-\bar{\lambda}\zeta}=\Abs{1-\exp(i\theta)\bar{\lambda}\Abs{\zeta}}$
and $\Norm{R(\zeta,\,T)}{}=\Norm{(\abs{\zeta}-\exp(-i\theta)T)^{-1}}{}$,
which entails that
\[
d(1,\,\overline{\sigma(T)}\zeta)\Norm{R(\zeta,\,T)}{}=d(1,\,\overline{\sigma(T')}\Abs{\zeta})\Norm{R(\Abs{\zeta},T')}{}
\]
with $T'=\exp(-i\theta)T$. Clearly $T\in\cC_{n}\Leftrightarrow T'\in\cC_{n}$ and $p(\zeta,\,\sigma(T))=p(\Abs{\zeta},\,\sigma(T'))$. The same argument implies $T\in\cT_{n}^a\Leftrightarrow T'\in\cT_{n}^a$
\end{proof}
The lemma says that for our optimization task it is sufficient to consider real $\zeta$.
\begin{proof}[Proof of Theorem~\ref{Kronecker_type_theorem}]
We consider a matrix $T\in\cC_{n}$ with $p(\zeta,\,\sigma(T))=r$ and the function $s$ in (\ref{eq:gen_upp_bd}). The function
$\lambda\mapsto\frac{1-\Abs{\lambda}^{2}}{\abs{1-\bar{\lambda}\abs{\zeta}}}$
is maximized on the pseudo-hyperbolic circle of center $\Abs{\zeta}$
and radius $r$ at $\lambda_{\textnormal{max}}=\frac{\Abs{\zeta}-r}{1-r\Abs{\zeta}}$ and takes the value $\beta_{{\rm max}}=\frac{1-\Abs{\lambda_{\max}}^{2}}{1-\bar{\lambda_{\max}}\Abs{\zeta}}=\frac{1-r^{2}}{1-r\Abs{\zeta}}$. Thus using Lemma~\ref{lemma_real_zeta} and Theorem~\ref{Th_upper_bound} we find that $\cR_{n}(\zeta,\, r)=\cR_{n}(\abs{\zeta},\, r)\leq\frac{1}{r^{n}}\Norm{X_{r,\,\beta_{{\rm max}}}}{}$.

According to the proof of Theorem~\ref{Th_upper_bound} the model operator corresponding to Blaschke product $B=b_{\lambda}^{n}$ with $\lambda\in(-1,1)$, $\lambda\neq\abs{\zeta}$ satisfies
\[
\Norm{R({\abs\zeta},M_{B})}{}=\frac{1}{d(1,\bar{\sigma}\abs\zeta)}\frac{1}{r^{n}}\Norm{X_{r,\,\beta}}{}.
\]
We choose $\lambda_{\max}=\frac{\Abs{\zeta}-r}{1-r\Abs{\zeta}}$ and
have $p(\Abs{\zeta},\,\sigma)=r$
and $s(\Abs{\zeta},\,\sigma)=\beta_{{\rm max}}=\frac{1-r^{2}}{1-r\Abs{\zeta}}$.
It follows that ${\cR_{n}}(\zeta,\, r)\geq\frac{1}{r^{n}}\Norm{X_{r,\,\beta}}{}$ and the \lq\lq worst\rq\rq{} matrix for the extremal problem in~\eqref{eq:Kr_Cst_1} is the analytic Toeplitz matrix $M_B$ with $B=b_{\lambda_{\max}}^{n}$.

The monotonicity properties and the asymptotic behaviour of $\cR_{n}(\zeta,\, r)$ are consequences of our study of $\Norm{X_{r,\,\beta_{{\rm max}}}}{}$, which is summarized in Proposition~\ref{Th_3_computation_of_norm_of_Xtilde}.
\end{proof}
\begin{proof}[Proof of Proposition \ref{cor:Davies_Simon_type_th}]
 We first prove \eqref{eq:th4_majoration}. We consider $T\in\cC_{n}$
with spectrum $\sigma$ and assume for a moment that $\sigma_{2}\subset\mathbb{D}$.
It follows from \cite[Theorem III.2]{OS} that for any $\zeta\in\mathbb{C}-\sigma$
the resolvent of $T$ is bounded by 
\[
\Norm{R(\zeta,\, T)}{}\leq\Norm{R(\zeta,\, M_{B})}{},
\]
where $B$ is the Blaschke product associated with $P_{\sigma}$.
We conclude from \eqref{eq:OS_formula} that the matrix of the resolvent
of the model operator $M_{B}$ at any point $\zeta\in\mathbb{C}-\sigma$
(with respect to the Malmquist-Walsh basis) is given by a block-diagonal
matrix
\[
\left(\begin{array}{cc}
A_{1} & 0\\
0 & A_{2}
\end{array}\right),
\]
where $A_{1}\in\cM_{\abs{P_{\sigma_{1}}}}(\mathbb{C})$ and $A_{2}\in\cM_{\abs{P_{\sigma_{2}}}}(\mathbb{C})$
are respectly defined by (taking now the limit as $\abs{\mu}\rightarrow1$
for $\mu\in\sigma_{2}$):
\[
(A_{1})_{1\leq i,\, j\leq\Abs{P_{\sigma_{1}}}}=\left\{ \begin{array}{ll}
0 & \mbox{if }i<j\\
\frac{1}{\zeta-\lambda_{i}} & \mbox{if }i=j\\
\frac{(1-\Abs{\lambda_{i}}^{2})^{1/2}}{1-\overline{\lambda_{i}}\zeta}\frac{(1-\Abs{\lambda_{j}}^{2})^{1/2}}{1-\overline{\lambda_{j}}\zeta}{\prod_{k=j}^{i}\frac{1-\bar{\lambda}_{k}\zeta}{\zeta-\lambda_{k}}} & \mbox{if }i>j
\end{array},\right.
\]
and
\[
A_{2}=\left(\begin{array}{ccc}
\frac{1}{\zeta-\mu_{1}}\\
 & \ddots\\
 &  & \frac{1}{\zeta-\mu_{n_{2}}}
\end{array}\right),
\]
where $\mu_{i}$ are the unimodular eigenvalues of $M_{B}$. To conclude,
it remains to remark that the spectral norm of $R(\zeta,\, T)$ is
nothing but 
\[
\max(\Norm{A_{1}}{},\,\Norm{A_{2}}{}),
\]
that $\Norm{A_{2}}{}=\frac{1}{d(\zeta,\,\sigma_{2})},$ while the
inequality 
\[
\Norm{A_{1}}{}\leq\frac{1}{d(\zeta,\,\sigma_{1})}\Norm{X_{1,\, s(\zeta,\sigma_{1})}}{},
\]
follows directly from Theorem \ref{Th_upper_bound}. This proves the
first inequality in \eqref{eq:th4_majoration}. The second one is
due to the fact that $s(\zeta,\sigma_{1})$ is bounded from above
by 2 and the equality refers to the resolvent estimate in \cite{DS}
or more generally to our Proposition~\ref{Th_3_computation_of_norm_of_Xtilde}. 

Now, let us prove \eqref{eq:th4_egalite}. We consider $T\in\cC_{n}$
with spectrum $\sigma=\sigma_{1}\cup\sigma_{2},$ $n_{1}=\abs{P_{\sigma_{1}}}$
and $n_{2}=\abs{P_{\sigma_{2}}}$ such that $\max_{\lambda\in\sigma_{1}}\abs{\lambda}\leq\rho_1.$
We fix $\zeta\in\partial\mathbb{D}-\sigma.$ The last inequality,
combined with the trivial observation that $s(\zeta,\sigma_{1})\leq\rho_1,$
proves that 
\[
\Norm{A_{1}}{}\leq\Norm{X_{1,\,\rho_1}}{}.
\]
To complete the proof of \eqref{eq:th4_egalite}, we first remark
that by rotation invariance (as in the proof of Lemma \ref{lemma_real_zeta})
the $\sup$ defined in the statement of \eqref{eq:th4_egalite} does
not depend on $\zeta$ and coincides with its value at point $\zeta=1.$
We consider the spectrum $\sigma=\sigma_{1}\cup\sigma_{2}$ defined
by setting $P_{\sigma_{1}}=(z-\rho_1)^{n_{1}}$ and $P_{\sigma_{2}}=(z+1)^{n_{2}}.$
Let $B$ be the Blaschke product associated with $P_{\sigma}.$ We
repeat the proof of \eqref{eq:th4_majoration} to the model operator
$M_{B}$. The matrix of its resolvent at point $1$ (with respect
to the Malmquist-Walsh basis) is given by the block-diagonal matrix
\[
\left(\begin{array}{cc}
A_{1} & 0\\
0 & A_{2}
\end{array}\right),
\]
where

\begin{align*}
(A_{1})_{1\leq i,\, j\leq n_{1}}&=\frac{1}{d(1,\,\sigma_{1})}\frac{1}{B({1})}\left\{ \begin{array}{ll}
0 & \mbox{if }i<j\\
(b_{\rho_1}({1}))^{n-1} & \mbox{if }i=j\\
\frac{{1-\abs{\rho_1}^{2}}}{\abs{1-\rho_1}}(b_{\rho_1}({1}))^{n-(i-j+1)} & \mbox{if }i>j
\end{array}\right.\\
&=\frac{X_{1,\,1+\rho_1}}{\abs{1-\rho_1}}.
\end{align*}
and $A_{2}=\frac{1}{2}I_{n_{2}},$ where $I_{n_{2}}$ is the identity
matrix of $\cM_{n_{2}}(\mathbb{C})$. Thus, 
\[
\Norm{R(1,\, M_{B})}{}=\max(\frac{1}{\abs{1-\rho_1}}\Norm{X_{1,\,1+\rho_1}}{},\,\frac{1}{2}),
\]
but $d(1,\,\sigma)=d(1,\,\sigma_{1})=\abs{1-\rho_1},$ and as a consequence:
\[
d(1,\,\sigma)\Norm{R(1,\, M_{B})}{}=\max(\Norm{X_{1,\,1+\rho_1}}{},\,\frac{\abs{1-\rho_1}}{2}).
\]
The above maximum is equal to $\Norm{X_{1,\,1+\rho_1}}{}:$ indeed,
according to Proposition~\ref{Th_3_computation_of_norm_of_Xtilde}, 
\[
\Norm{X_{1,\,1+\rho_1}}{}\geq\Norm{X_{1,\,1}}{}=\frac{1}{2\sin(\frac{\pi}{4n_{1}+2})},
\]
and the latter quantity $\frac{1}{2\sin(\frac{\pi}{4n_{1}+2})}$ is
always greater than $\frac{1}{2}$ (increasing with $n_{1})$ since
$n_{1}\geq1$. The last part of the statement relies again on Proposition~\ref{Th_3_computation_of_norm_of_Xtilde} and the fact that $\rho\mapsto\Norm{X_{1,\,1+\rho}}{}$
is increasing on $[0,\,1].$ 
\end{proof}
\begin{proof}[Proof of Proposition~\ref{Th_3_computation_of_norm_of_Xtilde}]

Our approach is guided by the methods developed in~\cite{EE,DS,OS}. Often it will be convenient to transform the Toeplitz matrix $X_{r,\,\beta}$ to an associated Hankel matrix $\widetilde{X}_{r,\,\beta}$, see~\cite{DS,OS}.
We set $\widetilde{X}_{r,\,\beta}=X_{r,\,\beta}J$, where
\begin{equation*}
J=\left(\begin{array}{ccc}
 &  & 1\\
 & \iddots    & \\
1 &    &\\
\end{array}\right).
\end{equation*}
$\tilde{X}$ is a Hermitian Hankel matrix that satisfies $\norm{\widetilde{X}}{}=\Norm{X_{r,\,\beta}}{}$. Hence, in our discussion we can focus on the largest in magnitude eigenvalue of $\tilde{X}$. (We omit indices in $\widetilde{X}$ for simpler notation.)

\emph{i)} In the case $\beta=1-r^2$ we have
\[
\frac{1}{r^{n}}X_{r,\,\beta}=\frac{1}{r^{n}}\left\{ \begin{array}{ll}
0 & \mbox{if }i<j\\
r{}^{n-1} & \mbox{if }i=j\\
(1-r^{2})r{}^{n-(i-j+1)} & \mbox{if }i>j
\end{array}\right.=\big(M_{b_{-r}^{n}}^{-1}\big)_{1\leq i,\, j\leq n},
\]
where the second equality is~\eqref{eq:OS_formula}. It is well known~\cite[Theorem 3.12, p.147]{NN1},~\cite[Lemma III.5]{OS} that (see the cited articles for a discussion)
\[
\Norm{M_{B}^{-1}}{}=\Norm{1/z}{H_{\infty}/BH_{\infty}}=\frac{1}{\vert B(0)\vert}=\prod_{\lambda\in\sigma}\frac{1}{\Abs{\lambda}}.
\]
In particular, we find
$\Norm{M_{b_{-r}^{n}}^{-1}}{}=\frac{1}{r^{n}}$ and $\Norm{X_{r,\,\beta}}{}=1$.

\emph{ii)} $\tilde{X}$ is an entry-wise non-negative matrix. A standard result in Perron-Froebenius theory~\cite[Theorem 4.2]{HM} asserts that the spectral radius of such matrices is a non-negative eigenvalue and the corresponding eigenvector is entry-wise non-negative\footnote{In fact, as can be seen easily by computing powers, $\tilde{X}$ is primitive.}. Furthermore the spectral radius is monotonically increasing in the matrix-entries~\cite[Corollary 2.1]{HM}, i.e. $\tilde{X}\leq B$ entry-wise implies $\rho(\tilde{X})\leq\rho(B)$.

\emph{iii)} Here, we show a simple proof for $\beta>1$. The more general discussion under the condition $1-r^2\leq \beta$ can be found below. Let $\tilde{X}_n\in\cM_n$ and
$\tilde{X}_{n+1}\in\cM_{n+1}$. We show $\norm{\tilde{X}_n}{}\leq\norm{\tilde{X}_{n+1}}{}$. Let $x^*$ with $\norm{x^*}{}=1$ and $\tilde{X}_nx^*=\norm{\tilde{X}_n}{}x^*$. Computing entries of vectors and using $\beta\geq1$ we see that
$$\Norm{\tilde{X}_{n}x^*}{}\leq\Norm{\tilde{X}_{n+1}\begin{pmatrix}x^*\\0\end{pmatrix}}{}.$$

\emph{iv)} We turn to the case $r^{2}+\beta-1\neq0$. The discussion is guided by~\cite{EE}. The eigenvalues of $\tilde{X}^{2}$ are the eigenvalues of $\tilde{X}$
squared. Hence, we are looking for the largest $\lambda^{2}$ such
that
\[
0=\det(\widetilde{X}^{2}-\lambda^{2})=
\det(\widetilde{X}-\lambda)\det(\widetilde{X}+\lambda).
\]
The matrix $\widetilde{X}^{2}-\lambda^{2}$ can be computed explicitly.
We display this matrix for $n=5$ for simplicity, see Appendix~\ref{sec:appendix-1} for the general matrix. If $n=5$ we have 
{\scriptsize{
\begin{align*}
&\textnormal{\normalsize{$\tilde{X}^{2}-\lambda^{2}$}}=\\
&\begin{pmatrix}-\lambda^{2}+\beta^{2}\frac{1-r^{2\cdot4}}{1-r^{2}}+r^{8} & \beta^{2}r\frac{1-r^{2\cdot3}}{1-r^{2}}+\beta r^{7} & \beta^{2}r^{2}\frac{1-r^{2\cdot2}}{1-r^{2}}+\beta r^{6} & \beta^{2}r^{3}\frac{1-r^{2\cdot1}}{1-r^{2}}+\beta r^{5} & \beta r^{4}\\
\beta^{2}r\frac{1-r^{2\cdot3}}{1-r^{2}}+\beta r^{7} & -\lambda^{2}+\beta^{2}r^{2}\frac{1-r^{2\cdot3}}{1-r^{2}}+r^{8} & \beta^{2}r^{3}\frac{1-r^{2\cdot2}}{1-r^{2}}+\beta r^{7} & \beta^{2}r^{4}\frac{1-r^{2\cdot1}}{1-r^{2}}+\beta r^{6} & \beta r^{5}\\
\beta^{2}r^{2}\frac{1-r^{2\cdot2}}{1-r^{2}}+\beta r^{6} & \beta^{2}r^{3}\frac{1-r^{2\cdot2}}{1-r^{2}}+\beta r^{7} & -\lambda^{2}+\beta^{2}r^{4}\frac{1-r^{2\cdot2}}{1-r^{2}}+r^{8} & \beta^{2}r^{5}\frac{1-r^{2\cdot1}}{1-r^{2}}+\beta r^{7} & \beta r^{6}\\
\beta^{2}r^{3}\frac{1-r^{2\cdot1}}{1-r^{2}}+\beta r^{5} & \beta^{2}r^{4}\frac{1-r^{2\cdot1}}{1-r^{2}}+\beta r^{6} & \beta^{2}r^{5}\frac{1-r^{2.1}}{1-r^{2}}+\beta r^{7} & -\lambda^{2}+\beta^{2}r^{6}\frac{1-r^{2\cdot1}}{1-r^{2}}+r^{8} & \beta r^{7}\\
\beta r^{4} & \beta r^{5} & \beta r^{6} & \beta r^{7} & -\lambda^{2}+r^{8}
\end{pmatrix}.
\end{align*}
}}%
The determinant of the above matrix is transformed by the following
steps.
\begin{enumerate}
\item We divide the $k^{th}$ column and the $k^{th}$ row of $(\widetilde{X}^{2}-\lambda^{2})$
by $r^{k-1}$ .

\item In the determinant resulting from (1), we subtract the $k^{th}$ column
from the $(k-1)^{th}$ and leave the $n^{th}$ unchanged.

\item In the determinant resulting from (2), we subtract the $(k-1)^{th}$ row
from the $k^{th}$ and leave the $n^{th}$ unchanged. We obtain
the determinant of a tri-diagonal $n\times n$ matrix.

\item In the determinant resulting from (3), we multiply the $k^{th}$ row
by ${r^{2k-1}}$. 
\end{enumerate}
We obtain that $\det(\tilde{X}^{2}-\lambda^{2})=0$ iff
\[
0=\det\left(\begin{array}{ccccc}
\gamma & \alpha' & 0 & \cdots & 0\\
\alpha & \gamma & \alpha' & \ddots & \vdots\\
0 & \alpha & \ddots & \ddots & 0\\
\vdots & \ddots & \ddots & \gamma & \alpha'\\
0 & \cdots & 0 & \alpha & y
\end{array}\right)=\det\left(\begin{array}{ccccc}
\gamma & \mu & 0 & \cdots & 0\\
\mu & \gamma & \mu & \ddots & \vdots\\
0 & \mu & \ddots & \ddots & 0\\
\vdots & \ddots & \ddots & \gamma & \mu\\
0 & \cdots & 0 & \mu & y
\end{array}\right),
\]
where $\alpha=\lambda^2r+r^{2n-1}(\beta-1),\ \alpha'=\frac{\alpha}{r^{2}},y=-r\lambda^2+r^{2n-1}$
and $\gamma=-\lambda^2(r+1/r)+r^{2n-2}(r+\frac{1}{r}(\beta-1)^{2})$ and we set $\mu=\alpha/r$ in the last step.
We have that 
\begin{multline*}
\det\left(\begin{array}{ccccc}
\gamma & \mu & 0 & \cdots & 0\\
\mu & \gamma & \mu & \ddots & \vdots\\
0 & \mu & \ddots & \ddots & 0\\
\vdots & \ddots & \ddots & \gamma & \mu\\
0 & \cdots & 0 & \mu & y
\end{array}\right)\\
=\det\left(\begin{array}{ccccc}
\gamma & \mu & 0 & \cdots & 0\\
\mu & \gamma & \mu & \ddots & \vdots\\
0 & \mu & \ddots & \ddots & 0\\
\vdots & \ddots & \ddots & \gamma & \mu\\
0 & \cdots & 0 & \mu & \gamma
\end{array}\right)+\det\left(\begin{array}{ccccc}
\gamma & \mu & 0 & \cdots & 0\\
\mu & \gamma & \mu & \ddots & \vdots\\
0 & \mu & \ddots & \ddots & 0\\
\vdots & \ddots & \ddots & \gamma & 0\\
0 & \cdots & 0 & \mu & y-\gamma
\end{array}\right)\\
=\det\left(\begin{array}{ccccc}
\gamma & \mu & 0 & \cdots & 0\\
\mu & \gamma & \mu & \ddots & \vdots\\
0 & \mu & \ddots & \ddots & 0\\
\vdots & \ddots & \ddots & \gamma & \mu\\
0 & \cdots & 0 & \mu & \gamma
\end{array}\right)+(y-\gamma)\det\left(\begin{array}{ccccc}
\gamma & \mu & 0 & \cdots & 0\\
\mu & \gamma & \mu & \ddots & \vdots\\
0 & \mu & \ddots & \ddots & 0\\
\vdots & \ddots & \ddots & \gamma & \mu\\
0 & \cdots & 0 & \mu & \gamma
\end{array}\right),
\end{multline*}
where in the last sum the first determinant is of order $n$, while
the second is of order $n-1$. The assumption $\mu=0$ is equivalent to
$\lambda^{2}=r^{2n-2}(1-\beta)$ and the determinant equation turns into $\gamma y=0$. In this case we obtain a solution if $\beta=0$ or $r^{2}+\beta-1=0$ but these choices of $\beta$ are treated differently. (Note that $\lambda=1$ implies $\mu\neq0$.) For the determinant of a tri-diagonal Toeplitz matrix we have the well known formula~\cite{MV,EP,EE,OS}
\[
\det\left(\begin{array}{ccccc}
a & 1 & 0 & \cdots & 0\\
1 & a & 1 & \ddots & \vdots\\
0 & 1 & \ddots & \ddots & 0\\
\vdots & \ddots & \ddots & a & 1\\
0 & \cdots & 0 & 1 & a
\end{array}\right)=U_n(a/2),
\]
where $U_n$ denotes the $n$-th degree Chebychev polynomial of the second kind. Hence the eigenvalues of $\tilde{X}$ correspond to the roots of the polynomial equation
$$\mu^{n+1} U_n\left(\frac{\gamma}{2\mu}\right)+\mu^n(y-\gamma)U_{n-1}\left(\frac{\gamma}{2\mu}\right)=0.$$

\emph{v)} We make use of properties of $U_n$ to reveal properties of the roots of equation~\eqref{eq:rrr}. Let us consider $\lambda^2$ so that there is \emph{real} $\theta$ with $\gamma/2\mu=\cos(\theta)$, i.e.
\begin{equation}
\lambda^{2}=r^{2(n-1)}\frac{(\beta-1)^{2}+r^{2}-2r(\beta-1)
\cos(\theta)}{1+r^{2}+2r\cos(\theta)}.\label{equ:222}
\end{equation}
%
%\[
%0=\mu^{n}\left(\frac{\sin((n+1)\theta)}{\sin(\theta)}+\frac{y-\gamma}{\mu}\frac{\sin(n\theta)}{\sin(\theta)}\right)
%\]
%
%
%
Using $U_n(\cos(\theta))=\frac{\sin{((n+1)\theta)}}{\sin{(\theta)}}$ Equation~\eqref{eq:rrr} turns into
\begin{equation*}
\frac{\sin((n+1)\theta)}{\sin(\theta)}+\frac{1}{r}\frac{\lambda^{2}-r^{2n-2}(\beta-1)^{2}}{\lambda^{2}+r^{2n-2}(\beta-1)}\frac{\sin(n\theta)}{\sin(\theta)}=0,
\end{equation*}
which plugging back $\lambda^{2}$ is
\begin{equation}
{\rm \cot}(n\theta)+\frac{1}{\tan(\theta)}\frac{r^{2}-(\beta-1)}{r^{2}+\beta-1}+\frac{(2-\beta)r}{r^{2}+\beta-1}\frac{1}{\sin(\theta)}=0.\label{eq:theta_equation_1-1}
\end{equation}
Similarly, if
$\gamma/2\mu\in[1,\infty)$ we set $\gamma/2\mu=\cosh(\theta)$ with $\theta\in[0,\infty)$, which leads to \cite[Section 1.4]{MH}
\begin{equation}
\lambda^{2}=r^{2(n-1)}
\frac{(\beta-1)^{2}+r^{2}-2r(\beta-1)\cosh(\theta)}{1+r^{2}+2r\cosh(\theta)}\label{equ:2221}
\end{equation}
with
\begin{equation}
\frac{\sinh{((n+1)\theta)}}{\sinh{(n\theta)}}=-\frac{(2-\beta)r+2(1-\beta)\cosh(\theta)}{r^2+\beta-1}.
\label{equ:2221b}
\end{equation}
Finally, in case that $\gamma/2\mu\in[-1,-\infty)$ we set $\gamma/2\mu=-\cosh(\theta)$ with $\theta\in[0,\infty)$ and recall that $U_n(-x)=(-1)^nU_n(x)$ \cite[Section 1.4]{MH} to find
\begin{equation}
\lambda^{2}=r^{2(n-1)}\frac{(\beta-1)^{2}+r^{2}+
2r(\beta-1)\cosh(\theta)}{1+r^{2}-2r\cosh(\theta)}\label{equ:2222}
\end{equation}
with
\begin{equation}
\frac{\sinh{((n+1)\theta)}}{\sinh{(n\theta)}}=\frac{(2-\beta)r-2(1-\beta)\cosh(\theta)}{r^2+\beta-1}.\label{equ:2222b}
\end{equation}

We consider the first case. Any solution $\theta_k$ of~\eqref{eq:theta_equation_1-1} leads to a solution $\lambda^2_k$ of $\det(\tilde{X}^{2}-\lambda^{2})=0$ via~\eqref{equ:222}. An important point about Equation~\eqref{equ:222} is that all such solutions $\abs{\lambda_k}^2$ can be bounded by a function that decays exponentially with $n$. This allows us to study properties of $\norm{\tilde{X}}{}$ by counting solutions of equation~\eqref{eq:theta_equation_1-1}. We observe that 
\begin{align*}&\frac{1}{\sin(\theta)}\to+\infty\quad\textnormal{as}\quad\theta\to0^+ &\textnormal{and}\qquad
\quad&\frac{1}{\tan(\theta)}\to+\infty\quad\textnormal{as}\quad\theta\to0^+
\\ &\frac{1}{\sin(\theta)}\to+\infty\quad\textnormal{as}\quad\theta\to\pi^- &\textnormal{and}\qquad\quad
&\frac{1}{\tan(\theta)}\to-\infty\quad\textnormal{as}\quad\theta\to\pi^-.
\end{align*}
It follows that if $1-r^2<\beta<1+r$ then (see Lemma~\ref{lem3} for details)
\begin{align*}
&\frac{1}{\tan(\theta)}\frac{r^{2}-(\beta-1)}{r^{2}+\beta-1}+\frac{(2-\beta)r}{r^{2}+\beta-1}\frac{1}{\sin(\theta)}\to+\infty\quad\textnormal{as}\quad\theta\to0^+\\
&\frac{1}{\tan(\theta)}\frac{r^{2}-(\beta-1)}{r^{2}+\beta-1}+\frac{(2-\beta)r}{r^{2}+\beta-1}\frac{1}{\sin(\theta)}\to+\infty\quad\textnormal{as}\quad\theta\to\pi^-.
\end{align*}
As $-\cot(n\theta)$ has $n$ branches in the interval $[0,\pi]$ and since
\begin{align*}
&-\cot(n\theta)\to-\infty\quad\textnormal{as}\quad\theta\to0^+\\
&-\cot(n\theta)\to+\infty\quad\textnormal{as}\quad\theta\to\pi^-,
\end{align*}
we conclude that Equation~\eqref{eq:theta_equation_1-1} has $n-1$ solutions $\theta_k\in(0,\pi)$ (if $n$ is large enough, see Lemma~\ref{lem3}). The corresponding eigenvalues $\lambda_k$ of $\tilde{X}$ approach $0$ exponentially fast, so that for the largest eigenvalue of $\tilde{X}$ we have 
$\lambda^*\to{\rm Tr}{(\tilde{X})}$. On the other hand direct consideration of the trace shows that at fixed $\beta$ and $r$ we have
\[
{\rm Tr}(\widetilde{X})\rightarrow\frac{\beta}{1-r^{2}}\quad\textnormal{as}\quad n\rightarrow\infty.
\]
If $\beta>1+r$ we have only $n-2$ real solutions $\theta_k$ to Equation~\eqref{eq:theta_equation_1-1}, see Lemma~\ref{lem3}. We consider Equations~\eqref{equ:2221},~\eqref{equ:2221b} and Equations~\eqref{equ:2222},~\eqref{equ:2222b}.  As $n$ gets large $\frac{\sinh{((n+1)\theta)}}{\sinh{(n\theta)}}$ quickly converges to $e^\theta$ with $\theta\in[0,\infty)$. By the intermediate value theorem \eqref{equ:2221b} has a solution for $\beta>1+r$ and sufficiently large $n$. Any $\lambda^2$ as in Equation~\eqref{equ:2221} decays exponentially with $n$. Hence, in case $\beta>1+r$ we obtain $n-1$ eigenvalues that decay exponentially with $n$, as before. In the asymptotic limit of large $n$ we find that $e^{(\theta^*)}=1/r$ becomes a solution of $\eqref{equ:2222b}$ and $\lambda^*$ in $\eqref{equ:2222}$ does \emph{not} decay exponentially with $n$, which identifies the latter as the largest in magnitude solution for large enough $n$. Hence, we have that $\Norm{X}{}\to \beta/(1-r^2)$ as $n\to\infty$.

\emph{iii) (second part)} We have already seen that $\norm{X_{r,\beta}}{}$ grows monotonically with $n$ if $\beta>1$. Hence, we consider $1-r^2<\beta<1$. As outlined in \emph{v)} there are for large enough\footnote{When there are $n$ solutions to~\eqref{eq:theta_equation_1-1}, monotonicity of $\norm{X_{r,\beta}}{}$ can be readily be seen in~\eqref{equ:222} and~\eqref{eq:theta_equation_1-1}.} $n$, see Lemma~\ref{lem3}, $n-1$ solutions $\theta_k$ to Equation~\eqref{eq:theta_equation_1-1}
and the corresponding eigenvalues, see~\eqref{equ:222}, can be bounded as
\begin{align}
r^{2(n-1)}\left(1-\frac{\beta}{1+r}\right)^2\leq\abs{\lambda_k}^2\leq r^{2(n-1)}\left(1-\frac{\beta}{1-r}\right)^2\label{eigvl}.
\end{align}
We write $\tilde{X}^{(n)}$ and $\tilde{X}^{(n-1)}$ for the matrix $\tilde{X}$ in $n$ and $n-1$ dimensions. Let $\Norm{\cdot}{FR}$ denote the Frobenius norm. Direct computation shows that
\begin{align*}
\Norm{\tilde{X}^{(n)}}{FR}^2-\Norm{\tilde{X}^{(n-1)}}{FR}^2=nr^{2(n-1)}+(n-1)\beta^2r^{2(n-2)}-(n-1)r^{2(n-2)}.
\end{align*}
On the other hand it is well known that $\norm{\tilde{X}^{(n)}}{FR}^2=\sum_{i=1}^n\abs{\lambda_i^{(n)}}^2$, where the $\lambda_i^{(n)}$ are the eigenvalues of $\tilde{X}^{(n)}$.  Writing $\lambda^{*,(n)}$ for the largest in magnitude eigenvalue of $\tilde{X}^{(n)}$ and 
making use the mentioned properties and the estimates~\eqref{eigvl} we find that
\begin{align*}
&\abs{\lambda^{*,(n)}}^2-\abs{\lambda^{*,(n-1)}}^2\\
&\geq nr^{2(n-1)}+(n-1)\beta^2r^{2(n-2)}-(n-1)r^{2(n-2)}\\
&\ -(n-1)r^{2(n-1)}\left(1-\frac{\beta}{1-r}\right)^2\\
&\ +(n-2)r^{2(n-2)}\left(1-\frac{\beta}{1+r}\right)^2.
\end{align*}
The lower bound is a degree-$2$ polynomial in $\beta$, which is convex if $r\in(0,1/2)$. Hence, the lower bound is an increasing function of $\beta\in[1-r^2,1]$. As $\beta=1-r^2$ implies $\abs{\lambda^{*,(n)}}^2-\abs{\lambda^{*,(n-1)}}^2=0$ this proves monotonicity in $n$ if $r\in(0,1/2)$. The range for $r$ can be extended by using better bounds in \eqref{eigvl}. Here we note that at the endpoints of the interval
for $\beta$ we already have established monotonicity for any $r$. Furthermore the norm is a monotonous, convex functions in $\beta$. Any of the branches of $-\cot(n\theta)$ is convex and the terms
$$\frac{r^{2}-(\beta-1)}{r^{2}+\beta-1}\quad\textnormal{and}\quad\frac{(2-\beta)r}{r^{2}+\beta-1}$$
in~\eqref{eq:theta_equation_1-1} are decreasing in $\beta\in[1-r^2,1]$.
It follows that the solutions $\theta_k^{(n)}$ of~\eqref{eq:theta_equation_1-1} satisfy
$$\theta_k^{(n)}(\beta_0)-\theta_k^{(n)}(\beta_1)\leq\theta_k^{(n-1)}(\beta_0)-\theta_k^{(n-1)}(\beta_1)$$
for $\beta_0\leq\beta_1$.
In conclusion the quantity
$$nr^{2(n-1)}+(n-1)\beta^2r^{2(n-2)}-(n-1)r^{2(n-2)}+\sum_{k=1}^{n-2}\abs{\lambda_k^{(n-1)}}^2
-\sum_{k=1}^{n-1}\abs{\lambda_k^{(n)}}^2
$$
is increasing in $\beta$.

\emph{vi)} is a direct consequence of \emph{iii)} and \emph{v)}.
\end{proof} 

\begin{proof}[Proof of Lemma~\ref{lem3}]
For fixed $\beta\in[0,\,2]$, $r\in(0,\,1]$ we for now focus our
attention on real values of $\theta$ and write shortly
\begin{align*}
\varphi_{r,\,\beta}(\theta)=\frac{r^{2}-(\beta-1)}{r^{2}+(\beta-1)}\cot(\theta)+\frac{(2-\beta)r}{\beta-1+r^{2}}\frac{1}{\sin(\theta)}
\end{align*}
such that Equation \eqref{eq:rrrr} becomes 
\[
\varphi_{r,\,\beta}(\theta)=-\cot(n\theta),\qquad\theta\in[-\pi,\,\pi).
\]
We start with some elementary observations. 
\begin{enumerate}
\item We restrict to the interval $[0,\,\pi)$ as all functions are symmetric
with respect to the point $(\pi,0)$. 
\item $\cot(n\theta)$ has poles in $[0,\,\pi)$ at $\theta_{k}=\frac{k}{n}\pi$,
$k=1,...,n-1$. 
\item The functions $\frac{1}{\sin(\theta)}$ and $\cot(\theta)$ have poles
at $0$ and $\pi$, $\cot(\theta)\sim\frac{1}{\sin(\theta)}$ as $\theta$
tends to $0^{+}$ and $\cot(\theta)\sim-\frac{1}{\sin(\theta)}$ as
$\theta$ tends to $\pi^{-}$. 
\item Looking at the polynomials: 
\[
P_{\beta}(r)=r^{2}-(2-\beta)r-(\beta-1)=(r-1)(r-(1-\beta)),
\]
and 
\[
Q_{\beta}(r)=r^{2}+(2-\beta)r-(\beta-1)=(r+1)(r-(\beta-1)),
\]
we have that $P_{\beta}(r)<0$ for $r\in(1-\beta,\,1)$ and $P_{\beta}(r)\geq0$
elsewhere, while $Q_{\beta}(r)<0$ for $r\in(-1,\,\beta-1)$ and $Q_{\beta}(r)\geq0$
elsewhere. 
\item Taking (c) and (d) into account, we have 
\[
\varphi_{r,\,\beta}(\theta)\sim_{_{\theta\rightarrow0^{+}}}\frac{r^{2}+(2-\beta)r-(\beta-1)}{\beta-1+r^{2}}\frac{1}{\sin(\theta)}=\frac{Q_{\beta}(r)}{\beta-1+r^{2}}\frac{1}{\sin(\theta)},
\]
while 
\[
\varphi_{r,\,\beta}(\theta)\sim_{_{\theta\rightarrow\pi^{-}}}-\frac{r^{2}-(2-\beta)r-(\beta-1)}{\beta-1+r^{2}}\frac{1}{\sin(\theta)}=\frac{-P_{\beta}(r)}{\beta-1+r^{2}}\frac{1}{\sin(\theta)}.
\]

\end{enumerate}
We now prove the points in the lemma. 
\begin{enumerate}

\item \textbf{The case $r=1.$} In this case, the structure of \eqref{eq:rrrr}
is discussed in \cite[Section B]{OS}. We now suppose that $r\in(0,\,1)$. 
\item \textbf{The case $r=\Abs{\beta-1}\notin\{0,\,2\}$.}

\begin{enumerate}
\item \textbf{The subcase $\beta<1$. } We have that 
\[
\varphi_{r,\,\beta}(\theta)={\displaystyle \frac{(-2+\beta)\,(\mathrm{cos}(\theta)+1)}{\mathrm{sin}(\theta)\,\beta}},
\]
and as a consequence, $\varphi_{r,\,\beta}<0$ on $(0,\,\pi)$, $\lim_{\theta\rightarrow0^{+}}\varphi_{r,\,\beta}(\theta)=-\infty$,
while $\lim_{_{\theta\rightarrow\pi^{-}}}\varphi_{r,\,\beta}(\theta)=0$.
Thus, the graph of $\varphi_{r,\,\beta}$ intersects $n-1$ times
with the graph of $\theta\mapsto-\cot(n\theta)$ on $[\frac{\pi}{2n},\,\pi)$.
Moreover, 
\[
\frac{\varphi_{r,\,\beta}(\theta)}{-\cot(n\theta)}=\frac{(2-\beta)\sin(\theta)\sin(n\theta)}{\beta\cos(n\theta)(1-\cos(\theta))},
\]
and in particular, for any $\theta\in(0,\,\frac{\pi}{2n})$, $\frac{\varphi_{r,\,\beta}(\theta)}{-\cot(n\theta)}>1$,
which can be seen by computing derivatives with respect to $\theta$:
$\theta\mapsto\frac{\varphi_{r,\,\beta}(\theta)}{-\cot(n\theta)}$
is increasing on $(0,\,\frac{\pi}{2n})$, with 
\[
\lim_{\theta\rightarrow0^{+}}\frac{\varphi_{r,\,\beta}(\theta)}{-\cot(n\theta)}=\frac{2n(2-\beta)}{\beta}\geq2n>1.
\]
The graph of $\varphi_{r,\,\beta}$ does not intersect the one of
$\theta\mapsto-\cot(n\theta)$ on $[0,\,\frac{\pi}{2n})$. We conclude
that in this case \eqref{eq:rrrr} has $2n-2$ real solutions $\theta_{k}$
in $[-\pi,\,\pi)$. 
\item \textbf{The subcase $\beta>1$. } We have that 
\[
\varphi_{r,\,\beta}(\theta)={\displaystyle \frac{(-2+\beta)\,(\mathrm{cos}(\theta)-1)}{\mathrm{sin}(\theta)\,\beta}},
\]
and as a consequence, $\varphi_{r,\,\beta}$ is increasing on $(0,\,\pi)$,
$\varphi_{r,\,\beta}>0$ on $(0,\,\pi)$, 
\[
\lim_{\theta\rightarrow0^{+}}\varphi_{r,\,\beta}(\theta)=0,
\]
while $\lim_{_{\theta\rightarrow\pi^{-}}}\varphi_{r,\,\beta}(\theta)=+\infty$.
Hence the graph of $\varphi_{r,\,\beta}$ intersects $n-1$ times
the one of $-\cot(n\theta)$ on $(0,\,\frac{(2n-1)\pi}{2n}]$. Moreover,
\[
\frac{\varphi_{r,\,\beta}(\theta)}{-\cot(n\theta)}=\frac{(2-\beta)\sin(\theta)\sin(n\theta)}{\beta\cos(n\theta)(1+\cos(\theta))},
\]
and in particular, for any $\theta\in(\frac{(2n-1)\pi}{2n},\,\pi)$,
$\frac{\varphi_{r,\,\beta}(\theta)}{-\cot(n\theta)}>\frac{2n(2-\beta)}{\beta}$
using the fact (as above) that $\theta\mapsto\frac{\varphi_{r,\,\beta}(\theta)}{-\cot(n\theta)}$
is this time decreasing on $(\frac{(2n-1)\pi}{2n},\,\pi)$, with 
\[
\lim_{\theta\rightarrow\pi^{-}}\frac{\varphi_{r,\,\beta}(\theta)}{-\cot(n\theta)}=\frac{2n(2-\beta)}{\beta}.
\]
Adding the positivity and the continuity of $\theta\mapsto\frac{\varphi_{r,\,\beta}(\theta)}{-\cot(n\theta)}$
on $(\frac{(2n-1)\pi}{2n},\,\pi)$, and 
\[
\lim_{\theta\rightarrow\frac{(2n-1)\pi}{2n}}\frac{\varphi_{r,\,\beta}(\theta)}{-\cot(n\theta)}=+\infty,
\]
the intermediate value theorem guarantees that the graph of $\varphi_{r,\,\beta}$
does not intersect the one of $\theta\mapsto-\cot(n\theta)$ on $(\frac{(2n-1)\pi}{2n},\,\pi)$
if and only if $\frac{2n(2-\beta)}{\beta}>1$ (i.e $n>\frac{\beta}{2(2-\beta)}$).
We conclude that in this case \eqref{eq:rrrr} has $2n-2$ real solutions
$\theta_{k}$ in $[-\pi,\,\pi)$ if and only if $n>\frac{\beta}{2(2-\beta)}$. 
\end{enumerate}

From now on, we suppose that $r\neq1-\beta$, and $r\neq\beta-1$.
Let us first notice that in case $\beta-1+r^{2}\neq0$, we have 
\[
\lim_{\theta\rightarrow0^{+}}\frac{\varphi_{r,\,\beta}(\theta)}{-\cot(n\theta)}=\frac{n(1+r)(\beta-1-r)}{\beta-1+r^{2}},
\]
and 
\[
\lim_{\theta\rightarrow\pi^{-}}\frac{\varphi_{r,\,\beta}(\theta)}{-\cot(n\theta)}=\frac{n(1-r)(r+\beta-1)}{\beta-1+r^{2}}.
\]
Below, we repeat the reasonings of (2) for the remaining cases (3),
(4) and (5).

\item \textbf{The case $\beta>1$, and $r\notin\{1,\,\beta-1\}$. }With
(e), $P_{\beta}(r)<0$, for any $r\in[0,\,1)$. In particular, for
any $r\in[0,\,1)$, 
\[
\lim_{_{\theta\rightarrow\pi^{-}}}\varphi_{r,\,\beta}(\theta)=+\infty,
\]
while 
\[
\lim_{\theta\rightarrow0^{+}}\varphi_{r,\,\beta}(\theta)=\left\{ \begin{array}{ll}
-\infty & \mbox{if }r\in[0,\,\beta-1)\\
+\infty & \mbox{\mbox{if }\ensuremath{r\in}\:(\ensuremath{\beta}-1,\,1)}
\end{array}\right.,
\]
because of the property of $Q_{\beta}(r)$ in (d).

\begin{enumerate}
\item \textbf{The subcase $1<\beta<1+r$}. In this case, $\varphi_{r,\,\beta}>0$
on $(0,\,\pi)$ and 
\[
\lim_{\theta\rightarrow0^{+}}\varphi_{r,\,\beta}(\theta)=\lim_{_{\theta\rightarrow\pi^{-}}}\varphi_{r,\,\beta}(\theta)=+\infty.
\]
Moreover, $\varphi_{r,\,\beta}$ has a global minimum at some point
$\theta_{0}\in(0,\,\pi)$ such that $\varphi_{r,\,\beta}$ is decreasing
on $(0,\,\theta_{0}]$ and increasing on $[\theta_{0},\,\pi)$. Thus,
the curve of $\varphi_{r,\,\beta}$ intersects the one of $\theta\mapsto-\cot(n\theta)$
exactly $n-1$ times on $(0,\,\frac{(2n-1)\pi}{2n}]$. On the interval
$(\frac{(2n-1)\pi}{2n},\,\pi$), the curve of $\varphi_{r,\,\beta}$
does not intersect the one of $\theta\mapsto-\cot(n\theta)$ if and
only if $\lim_{\theta\rightarrow\pi^{-}}\frac{\varphi_{r,\,\beta}(\theta)}{-\cot(n\theta)}>1$,
(else, the number of real intersections on $[0,\,\pi)$ is $n$).
Indeed as before, $\theta\mapsto\frac{\varphi_{r,\,\beta}(\theta)}{-\cot(n\theta)}$
is (positive, continuous and) decreasing on $(\frac{(2n-1)\pi}{2n},\,\pi)$.
We conclude that in this case \eqref{eq:rrrr} has $2n-2$ real solutions
$\theta_{k}$ in $[-\pi,\,\pi)$ if and only if $n>\frac{\beta-1+r^{2}}{(1-r)(r+\beta-1)}$. 
\item \textbf{The subcase $\beta>1+r.$ }In this case, $\varphi_{r,\,\beta}$
is increasing on $(0,\,\pi)$, and $\lim_{\theta\rightarrow0^{+}}\varphi_{r,\,\beta}(\theta)=-\infty$,
while $\lim_{_{\theta\rightarrow\pi^{-}}}\varphi_{r,\,\beta}(\theta)=+\infty$.
Thus, the curve of $\varphi_{r,\,\beta}$ intersects the one of $\theta\mapsto-\cot(n\theta)$
exactly $n-2$ times on $[\frac{\pi}{2n},\,\frac{(2n-1)\pi}{2n}]$.
Moreover, on $(0,\,\frac{\pi}{2n})\cup(\frac{(2n-1)\pi}{2n},\,\pi)$,
$\theta\mapsto\frac{\varphi_{r,\,\beta}(\theta)}{-\cot(n\theta)}$
is positive and continuous; it is increasing on $(0,\,\frac{\pi}{2n})$
and decreasing on $(\frac{(2n-1)\pi}{2n},\,\pi)$. Thus, the curve
of $\varphi_{r,\,\beta}$ does not intersect the one of $\theta\mapsto-\cot(n\theta)$
on $(0,\,\frac{\pi}{2n})$ (resp. on $(\frac{(2n-1)\pi}{2n},\,\pi))$
if and only if $\lim_{\theta\rightarrow0^{+}}\frac{\varphi_{r,\,\beta}(\theta)}{-\cot(n\theta)}>1$
(resp. $\lim_{\theta\rightarrow\pi^{-}}\frac{\varphi_{r,\,\beta}(\theta)}{-\cot(n\theta)}>1$).
We conclude that in this case \eqref{eq:rrrr} has $2n-4$ real solutions
$\theta_{k}$ in $[-\pi,\,\pi)$ if and only if $n>\frac{\beta-1+r^{2}}{(1-r)(r+\beta-1)}$, since  $\frac{\beta-1+r^{2}}{(1+r)(\beta-1-r)}<\frac{\beta-1+r^{2}}{(1-r)(r+\beta-1)}$.
\end{enumerate}
\item \textbf{The case $\beta<1$ and $r\notin\{1,\,1-\beta\}$. }With (e),
$Q_{\beta}(r)\geq0$, for any $r\in[0,\,1)$. In particular, for any
$r\in[0,\,1)$, 
\[
\lim_{_{\theta\rightarrow0^{+}}}\varphi_{r,\,\beta}(\theta)=\left\{ \begin{array}{ll}
+\infty & \mbox{if }\beta-1+r^{2}>0\\
-\infty & \mbox{\mbox{if }}\beta-1+r^{2}<0
\end{array}\right.,
\]
while 
\[
\lim_{\theta\rightarrow\pi^{-}}\varphi_{r,\,\beta}(\theta)=\left\{ \begin{array}{ccc}
+\infty & \mbox{if }\:\beta-1+r^{2}>0\\
+\infty & \mbox{if }\:\beta-1+r^{2}<0 & \:{\rm and}\: r\in[0,\,1-\beta)\\
-\infty & \mbox{if }\:\beta-1+r^{2}<0 & \:{\rm and}\:\ensuremath{r\in}(1-\ensuremath{\beta},\,1)
\end{array},\right.
\]
because of the property of $P_{\beta}(r)$ in (d).

\begin{enumerate}
\item \textbf{The subcase $1-r^{2}<\beta<1.$ }In this case, $\varphi_{r,\,\beta}>0$
on $(0,\,\pi)$, and $\lim_{\theta\rightarrow0^{+}}\varphi_{r,\,\beta}(\theta)=\lim_{_{\theta\rightarrow\pi^{-}}}\varphi_{r,\,\beta}(\theta)=+\infty$.
Moreover, $\varphi_{r,\,\beta}$ has a global minimum at some point
$\theta_{0}\in(0,\,\pi)$ such that $\varphi_{r,\,\beta}$ is decreasing
on $(0,\,\theta_{0}]$ and increasing on $[\theta_{0},\,\pi)$. We
conclude as in (3)-(a) that in this case, \eqref{eq:rrrr} has $2n-2$
real solutions $\theta_{k}$ in $[-\pi,\,\pi)$ if and only if $n>\frac{\beta-1+r^{2}}{(1-r)(r+\beta-1)}$. 
\item \textbf{The subcase $0<\beta<1-r^{2}.$ }

\begin{enumerate}
\item \textit{If} $0<\beta<1-r,$ $\varphi_{r,\,\beta}$ is increasing on
$(0,\,\pi)$, $\lim_{_{\theta\rightarrow0^{+}}}\varphi_{r,\,\beta}(\theta)=-\infty$,
and $\lim_{\theta\rightarrow\pi^{-}}\varphi_{r,\,\beta}(\theta)=+\infty$.
We conclude by a similar reasoning as the one in (3)-(b) that in this
case also \eqref{eq:rrrr} has $2n-4$ real solutions $\theta_{k}$
in $[-\pi,\,\pi)$ if and only if $n>\frac{\beta-1+r^{2}}{(1-r)(r+\beta-1)}$, since  $\frac{\beta-1+r^{2}}{(1+r)(\beta-1-r)}<\frac{\beta-1+r^{2}}{(1-r)(r+\beta-1)}$.

\item \textit{If $1-r<\beta<1-r^{2},$} $\varphi_{r,\,\beta}(\theta)<0,\:\forall\theta\in(0,\,\pi)$,
and $\lim_{\theta\rightarrow0^{+}}\varphi_{r,\,\beta}(\theta)=\lim_{_{\theta\rightarrow\pi^{-}}}\varphi_{r,\,\beta}(\theta)=-\infty$.
Moreover, $\varphi_{r,\,\beta}$ has a global maximum $\theta_{0}\in(0,\,\pi)$
such that $\varphi_{r,\,\beta}$ is increasing on $(0,\,\theta_{0}]$
and decreasing on $[\theta_{0},\,\pi)$. Thus, the curve of $\varphi_{r,\,\beta}$
intersects the one of $\theta\mapsto-\cot(n\theta)$ exactly $n-1$
times on $[\frac{\pi}{2n},\,\pi)$. in the interval $(0,\,\frac{\pi}{2n}$),
the curve of $\varphi_{r,\,\beta}$ does not intersect the one of
$\theta\mapsto-\cot(n\theta)$ if and only if $\lim_{\theta\rightarrow0^{+}}\frac{\varphi_{r,\,\beta}(\theta)}{-\cot(n\theta)}>1$.
The latter condition is always satisfied since $\frac{n(1+r)(\beta-1-r)}{\beta-1+r^{2}}=\frac{n(1+r)(1-\beta+r)}{1-\beta-r^{2}}>n$.
We conclude that in this case, \eqref{eq:rrrr} has $2n-2$ real solutions
$\theta_{k}$ on $[-\pi,\,\pi)$. 
\end{enumerate}
\end{enumerate}
\item \textbf{The case $\beta=1$}. In this case, 
\[
\varphi_{r,\,\beta}(\theta)={\displaystyle \frac{1+r\mathrm{cos}(\theta)}{r\mathrm{sin}(\theta)}},
\]
and thus $\varphi_{r,\,\beta}(\theta)>0,\:\forall\theta\in(0,\,\pi)$,
and $\lim_{\theta\rightarrow0^{+}}\varphi_{r,\,\beta}(\theta)=\lim_{_{\theta\rightarrow\pi^{-}}}\varphi_{r,\,\beta}(\theta)=+\infty$.
Moreover, $\varphi_{r,\,\beta}$ has a global minimum at some point
$\theta_{0}\in(0,\,\pi)$ such that $\varphi_{r,\,\beta}$ is decreasing
on $(0,\,\theta_{0}]$ and increasing on $[\theta_{0},\,\pi)$. In
particular, the graph of $\varphi_{r,\,\beta}$ intersects the one
of $\theta\mapsto-\cot(n\theta)$ exactly $n-1$ times on $(0,\,\frac{(2n-1)\pi}{2n}]$.
On the interval $(\frac{(2n-1)\pi}{2n},\,\pi$), the function $\theta\mapsto\frac{\varphi_{r,\,\beta}(\theta)}{-\cot(n\theta)}$
is decreasing and thus the curve of $\varphi_{r,\,\beta}$ does not
intersect the one of $\theta\mapsto-\cot(n\theta)$ if and only if
$\lim_{\theta\rightarrow\pi^{-}}\frac{\varphi_{r,\,\beta}(\theta)}{-\cot(n\theta)}>1$.
We conclude that in this case, \eqref{eq:rrrr} has $2n-2$ real solutions
$\theta_{k}$ on $[-\pi,\,\pi)$ if and only if $n>\frac{r}{1-r}$
(else \eqref{eq:rrrr} has $2n$ real solutions in $[-\pi,\,\pi)$). 
\end{enumerate}
\end{proof}

\section{Appendix \label{sec:appendix-1}}

The columns $C_{i},\: i=1,\dots,n$ of the $n\times n$ matrix $\widetilde{X}-\lambda^{2}$ are:

\[
C_{1}=\left[\begin{array}{c}
-\lambda^{2}+\beta^{2}(1+\, r^{2}+\ldots+r^{2.(n-2)})+r^{2.(n-1)}\\
\beta^{2}\, r(1+r^{2}+r^{2.(n-3)})+\beta\, r^{2.(n-1)-1}\\
\beta^{2}\, r^{2}(1+r^{2}+r^{2.(n-4)})+\beta\, r^{2.(n-1)-2}\\
\vdots\\
\beta^{2}\, r^{n-3}(1+r^{2})+\beta\, r^{n+1}\\
\beta^{2}\, r^{n-2}+\beta\, r^{n}\\
\beta\, r^{n-1}
\end{array}\right],
\]

\[
C_{2}=\left[\begin{array}{c}
\beta^{2}\, r(1+r^{2}+r^{2.(n-3)})+\beta\, r^{2.(n-1)-1}\\
-\lambda^{2}+\beta^{2}r^{2}(1+\, r^{2}+\ldots+r^{2.(n-3)})+r^{2.(n-1)}\\
\beta^{2}\, r^{3}(1+r^{2}+r^{2.(n-4)})+\beta\, r^{2.(n-1)-1}\\
\vdots\\
\beta^{2}\, r^{n-2}(1+r^{2})+\beta\, r^{n+2}\\
\beta^{2}\, r^{n-1}+\beta\, r^{n+1}\\
\beta\, r^{n}
\end{array}\right],
\]

\[
C_{3}=\left[\begin{array}{c}
\beta^{2}\, r^{2}(1+r^{2}+r^{2.(n-4)})+\beta\, r^{2.(n-1)-2}\\
\beta^{2}\, r^{3}(1+r^{2}+r^{2.(n-4)})+\beta\, r^{2.(n-1)-1}\\
-\lambda^{2}+\beta^{2}r^{4}(1+\, r^{2}+\ldots+r^{2.(n-4)})+r^{2.(n-1)}\\
\beta^{2}\, r^{5}(1+\, r^{2}+\ldots+r^{2.(n-5)})+\beta\, r^{2.(n-1)-1}\\
\vdots\\
\beta^{2}\, r^{n-1}(1+r^{2})+\beta\, r^{n+3}\\
\beta^{2}\, r^{n}+\beta\, r^{n+2}\\
\beta\, r^{n+1}
\end{array}\right],
\]

\[
C_{4}=\left[\begin{array}{c}
\beta^{2}\, r^{3}(1+r^{2}+r^{2.(n-5)})+\beta\, r^{2.(n-1)-3}\\
\beta^{2}\, r^{4}(1+r^{2}+r^{2.(n-5)})+\beta\, r^{2.(n-1)-2}\\
\beta^{2}\, r^{5}(1+r^{2}+r^{2.(n-5)})+\beta\, r^{2.(n-1)-1}\\
-\lambda^{2}+\beta^{2}r^{6}(1+\, r^{2}+\ldots+r^{2.(n-5)})+r^{2.(n-1)}\\
\beta^{2}\, r^{7}(1+r^{2}+r^{2.(n-6)})+\beta\, r^{2.(n-1)-1}\\
\vdots\\
\beta^{2}\, r^{n}(1+r^{2})+\beta\, r^{n+4}\\
\beta^{2}\, r^{n+1}+\beta\, r^{n+3}\\
\beta\, r^{n+2}
\end{array}\right],\dots
\]

\[
\dots C_{n-3}=\left[\begin{array}{c}
\beta^{2}\, r^{n-4}(1+r^{2}+r^{4})+\beta\, r^{n+2}\\
\beta^{2}\, r^{n-3}(1+r^{2}+r^{4})+\beta\, r^{n+3}\\
\beta^{2}\, r^{n-2}(1+r^{2}+r^{4})+\beta\, r^{n+4}\\
\vdots\\
\beta^{2}\, r^{2n-10}(1+r^{2}+r^{4})+r^{2n-4}\\
\beta^{2}\, r^{2n-9}(1+r^{2}+r^{4})+r^{2n-3}\\
-\lambda^{2}+\beta^{2}\, r^{2n-8}(1+r^{2}+r^{4})+r^{2n-2}\\
\beta^{2}\, r^{2n-7}(1+r^{2})+\beta\, r^{2n-3}\\
\beta^{2}\, r^{2n-6}+\beta\, r^{2n-4}\\
\beta\, r^{2n-5}
\end{array}\right],
\]
\[
C_{n-2}=\left[\begin{array}{c}
\beta^{2}\, r^{n-3}(1+r^{2})+\beta\, r^{n+1}\\
\beta^{2}\, r^{n-2}(1+r^{2})+\beta\, r^{n+2}\\
\vdots\\
\vdots\\
\beta^{2}\, r^{2n-9}(1+r^{2})+\beta\, r^{2n-5}\\
\beta^{2}\, r^{2n-8}(1+r^{2})+\beta\, r^{2n-4}\\
\beta^{2}\, r^{2n-7}(1+r^{2})+\beta\, r^{2n-3}\\
-\lambda^{2}+\beta^{2}\, r^{2n-6}(1+r^{2})+r^{2n-2}\\
\beta^{2}\, r^{2n-5}+\beta\, r^{2n-3}\\
\beta\, r^{2n-4}
\end{array}\right],
\]
\[
C_{n-1}=\left[\begin{array}{c}
\beta^{2}r^{n-2}+\beta r^{n}\\
\beta^{2}r^{n-1}+\beta r^{n+1}\\
\beta^{2}r^{n}+\beta r^{n+2}\\
\vdots\\
\vdots\\
\vdots\\
\vdots\\
\beta^{2}\, r^{2n-5}+\beta\, r^{2n-3}\\
-\lambda^{2}+\beta^{2}\, r^{2n-4}+r^{2n-2}\\
\beta\, r^{2n-3}
\end{array}\right],
\]
and

\[
C_{n}=\left[\begin{array}{c}
\beta r^{n-1}\,\\
\beta r^{n}\,\\
\beta r^{n+1}\,\\
\vdots\\
\vdots\\
\vdots\\
\vdots\\
\beta\, r^{2n-4}\\
\beta\, r^{2n-3}\\
-\lambda^{2}+r^{2n-2}
\end{array}\right].
\]

$\,$

\end{document}